\documentclass[11pt,leqno]{amsart}

\usepackage{graphicx}
\usepackage{amsfonts,delarray,amssymb,amsmath, color, amsthm,a4,a4wide}

\usepackage{amsmath,amsthm,amssymb,amsfonts,enumerate,color,esint,graphicx, float}

\usepackage{latexsym}
\usepackage{epsfig}
\usepackage{color}
\usepackage{esint}

\usepackage[margin=0.94in]{geometry}

\newcommand{\dist}{\text{dist}} 

\newcommand{\diam}{\text{diam}}

\newcommand{\trace}{\text{trace}}

\newcommand{\R}{{\mathbb R}} 

\newcommand{\Z}{{\mathbb Z}}
\newcommand{\T}{{\mathbb T}}

\newcommand{\e}{\varepsilon}

\newcommand{\eps}{\epsilon}

\newcommand{\p}{\partial}

\renewcommand{\div}{\mbox{div}\,}

\newcommand{\norm}[1]{\left\Vert#1\right\Vert}

\newcommand{\dx}{\, dx}
\DeclareMathOperator*{\osc}{osc}

\newenvironment{myindentpar}[1]%
{\begin{list}{}%
         {\setlength{\leftmargin}{#1}}%
         \item[]%
}
{\end{list}}

\theoremstyle{plain}
\newtheorem{thm}{Theorem}[section]
\newtheorem{cor}[thm]{Corollary}
\newtheorem{lem}[thm]{Lemma}

\newtheorem{prop}[thm]{Proposition}
\newtheorem{rem}[thm]{Remark}

\theoremstyle{definition}
\newtheorem{defn}[thm]{Definition}

\makeatletter
\def\l@subsection{\@tocline{2}{0pt}{2.5pc}{5pc}{}}
\makeatother

\numberwithin{equation}{section}

\begin{document}
  
\title[Semigeostrophic equations and linearized Monge-Amp\`ere equation]{ 
H\"older regularity of the 2D dual semigeostrophic equations via analysis of linearized Monge-Amp\`ere equations
}
\date{}
\author{Nam Q. Le}
\address{Department of Mathematics, Indiana University, Bloomington, IN 47405, USA}
\email{nqle@indiana.edu}
\thanks{The research of the author was supported in part by NSF grant DMS-1500400.}
\subjclass[2000]{35B65, 35J08, 35J96, 35J70, 35Q35, 35Q86, 86A10}
\keywords{\tiny Semigeostrophic equations, Linearized Monge-Amp\`ere equation, Polar factorization, H\"older estimate}

\begin{abstract} 
We obtain the H\"older regularity of time derivative of solutions to
the dual semigeostrophic equations in two dimensions when the initial potential density is bounded away from zero and infinity. 
Our main tool is an interior H\"older estimate in two dimensions for an inhomogeneous linearized Monge-Amp\`ere equation with right hand side being the divergence
of a bounded vector field. As a further application of our H\"older estimate,
we prove the H\"older regularity of the polar factorization for time-dependent maps in two dimensions with densities bounded away from zero and infinity. 
Our applications improve previous work by G. Loeper who considered the cases of densities sufficiently close to a positive constant.
\end{abstract}

\maketitle


\setcounter{equation}{0}

\section{Introduction and statement of the main results}
In this paper, we obtain the H\"older regularity of time derivative of solutions to
the dual semigeostrophic equations in two dimensions when the initial potential density is bounded away from zero and infinity; see Theorem \ref{GSthm}. 
Our main tool is an interior H\"older estimate in two dimensions for an inhomogeneous linearized Monge-Amp\`ere equation with right hand side being the divergence
of a bounded vector field
when the Monge-Amp\`ere measure is only assumed to be bounded between two positive 
constants; see Theorem \ref{Holder_int_thm}. As a further application of our H\"older estimate,
we prove the H\"older regularity of the polar factorization for time-dependent maps in two dimensions with densities bounded away from zero and infinity; see Theorems \ref{H_thm_1} and \ref{peri_thm}. 
Our applications improve previous work by Loeper \cite{Loe} who considered the cases of densities sufficiently close to a positive constant.
\subsection{The dual semigeostrophic equations on $\T^2$}
\label{SG_sec1}
The semigeostrophic equations are a simple model used in meteorology to describe
large scale atmospheric flows. As explained for example in Benamou-Brenier \cite[Section 2.2]{BB},
Loeper \cite[Section 1.1]{Loe2}, and Cullen \cite{Cu}, the semigeostrophic
equations can be derived from the three-dimensional incompressible Euler equations, with
Boussinesq and hydrostatic approximations, subject to a strong Coriolis force. Since
for large scale atmospheric flows the Coriolis force dominates the advection term,
the flow is mostly bi-dimensional. 

Here we focus on the dual semigeostrophic equations. Note that, using solutions to the dual equations together with the $W^{2,1}$ regularity
for Aleksandrov solutions to the Monge-Amp\`ere equations obtained by De Philippis and Figalli \cite{DPF}, Ambrosio-Colombo-De Philippis-Figalli \cite{ACDF12} established global in time distributional solutions to the 
original semigeostrophic equations on the two dimensional torus. For more on the Monge-Amp\`ere equations and Aleksandrov solutions, see the books by Figalli \cite{Fi} and Guti\'errez \cite{G01}.

The {\it dual equations} 
of the semigeostrophic 
equations on the two dimensional torus $\T^2=\R^2/\Z^2$ are the following system of nonlinear transport equations
\begin{equation}\label{GS:dual}
\left\{\begin{array}{rll}
\partial_t \rho_t(x) + \div (\rho_t(x) U_t(x)) &=0 & (t,x)\in [0,\infty)\times\T^2,\\
U_t(x) &= (x-\nabla P_t^{\ast}(x) )^\perp & (t,x)\in [0,\infty)\times\T^2,\\
\det D^2 P_t^{\ast}(x) &= \rho_t(x)& (t,x)\in [0,\infty)\times\T^2,\\
P_t^{\ast}(x) \text{ convex }&   & (t,x)\in \T^2,\\
\rho_0(x)&= \rho^0(x)& x\in \T^2
\end{array}\right.
\end{equation}
for $(\rho_t, P_t^{\ast})$
with the boundary condition
\begin{equation}
\label{SGbdr}
 P_t^{\ast}-|x|^2/2 ~\text{is}~ \Z^2- \text{periodic}.
\end{equation}
Here the initial potential density $\rho^0$ is a probability measure on $\T^2$.
Throughout, we use $w^\perp$ to denote the rotation by $\pi/2$ vector $(-w_2, w_1)$ for $w = (w_1, w_2)\in\R^2$ and $f_t(\cdot)$ to denote the function $f(t,\cdot)$.

Existence of global weak solutions for the (\ref{GS:dual})-(\ref{SGbdr}) system has been established via time discretization in Benamou-Brenier \cite{BB} and Cullen-Gangbo \cite{CG}. To be precisely, in 
these cited papers, the proof is given in $\R^3$, but it can be rewritten verbatim on the two-dimensional torus by using the optimal transport maps; see 
\cite[Theorems 2.1 and 3.1]{ACDF12} for further details.
When $\rho^0$ is H\"older continuous and bounded away from zero and infinity on $\T^2$, Loeper \cite{Loe2} showed that there is a unique, short-time, H\"older solution $\rho$ to 
(\ref{GS:dual})-(\ref{SGbdr}); the time interval for this H\"older
solution depends only on the bounds on $\rho^0$.
However, when $\rho^0$ is only a general probability measure, the uniqueness of weak solutions is still an open question. Due to this lack of uniqueness 
and to avoid unnecessary confusions,
we make the 
following definition on weak solutions (as already established in \cite{BB} and \cite{CG}) to 
(\ref{GS:dual})-(\ref{SGbdr}) that we are going to use throughout the paper.
\begin{defn}
By a weak solution  to (\ref{GS:dual})-(\ref{SGbdr}), we mean a pair of functions $(\rho_t, P_t^\ast)$ on $\R^2$ with the following properties:
\begin{myindentpar}{1cm}
(i) $P_t^\ast$ is convex on $\R^2$ with $P_t^{\ast}-|x|^2/2$ being $\Z^2$ periodic; $\rho_t$ is $\Z^2$ periodic;\\
(ii) $P_t^\ast$ is an Aleksandrov solution to the Monge-Amp\`ere equation
$$\det D^2 P_t^\ast=\rho_t  \text{ in }\T^2.$$
(iii) $U_t(x) = (x-\nabla P_t^{\ast}(x) )^\perp$ and $\rho_t$ satisfy the equations $\partial_t \rho_t(x) + \div (\rho_t(x) U_t(x)) =0$ and $\rho_0=\rho^0$ on $\T^2$ in the distributional sense, that is, 
$$\displaystyle\int\int_{\T^2} \left\{\p_t \varphi_t(x) + \nabla \varphi_t(x) \cdot U_t(x)\right\}\rho_t(x) dx dt +\int_{\T^2} \varphi_0(x)\rho^0(x) dx =0$$
for every $\varphi\in C_{0}^{\infty}([0,\infty)\times \R^2)$ $\Z^2$-periodic in the space variable.
\end{myindentpar}
\end{defn}

For completeness, we briefly indicate how to obtain distributional solutions to the original semigeostrophic equations from solutions $(\rho_t, P_t^\ast)$
of the dual equations (\ref{GS:dual})-(\ref{SGbdr}); see \cite{ACDF12} for a rigorous treatment. 
Let us denote by $P_t$ the Legendre transform of $P_t^{\ast}$, that is,
$$P_t(x) =\sup_{y\in \R^2}(x\cdot y-P_t^{\ast}(y)).$$
Let $p^0(x)= P_0(x)-|x|^2/2$ and 
\begin{equation}
 \left\{
 \begin{array}{rl}
  p_t (x) &:=P_t(x)-|x|^2/2,\\
  u_t (x) &:= (\p_t \nabla P_t^{\ast})\circ \nabla P_t(x) + D^2 P_t^{\ast}(\nabla P_t(x))\cdot (\nabla P_t(x)-x)^{\perp}.
 \end{array}
\right.
\end{equation}
Then $(p_t, u_t)$ is a global Eulerian solution to the original semigeostrophic equations:
\begin{equation}
\label{SG_ori}
\left\{\begin{array}{rll}
\partial_t \nabla p_t(x) + (u_t(x)\cdot\nabla) \nabla p_t(x)  -(\nabla p_t(x))^{\perp} + u_t(x) &= 0 & (t,x)\in [0,\infty)\times\T^2,\\
\div u_t(x) &= 0& (t,x)\in [0,\infty)\times\T^2,\\
p_0(x)&= p^0(x)& x\in \T^2.
\end{array}\right.
\end{equation}
In (\ref{SG_ori}), the functions $u_t$ and $p_t$ represent respectively the {\it velocity} and the {\it pressure}. The quantity $u_t^g$ related to the system (\ref{SG_ori}) defined by
$u_t^g(x) = (\nabla p_t(x))^\perp$
is called the {\it semi-geostrophic} wind.

We now return to the regularity of solutions to (\ref{GS:dual})-(\ref{SGbdr}) in the typical case where the initial density $\rho^0$ is bounded between two positive constants $\lambda$ and $\Lambda$.
The space regularity of the solutions is now well understood thanks to regularity results for the Monge-Amp\`ere equations which are mainly due to Caffarelli, De Philippis, Figalli, Savin, and Schmidt \cite{C,C1, C2, C3, DPF, DPFS, Sch}. We will recall these results in Theorems \ref{thm_collect} and \ref{Cthm}. 

Regarding the regularity with respect to time, to the best of our knowledge, the most refined result so far is due to Loeper \cite{Loe} under the condition that $\lambda$ and $\Lambda$ are close. More precisely, Loeper shows that if the initial potential density $\rho^0$ is sufficiently close to a positive constant, say,
$1-\e_0\leq\rho^0\leq 1+ \e_0$ on $\T^2$ for some $\e_0>0$ small,
 then $\p_t P_t,\p_t P_t^{\ast} \in L^{\infty}((0,\infty), C^{\alpha_0}(\T^2))$ where $\alpha_0>0$ depends only on $\e_0$; see \cite[Theorems 2.2, 2.3 and 9.2]{Loe}.

It is thus an interesting problem to study the H\"older continuity of $\p_t P_t^{\ast}$ and $\p_t P_t$ in the system (\ref{GS:dual})-(\ref{SGbdr})
when 
the closeness of the density $\rho^0$
to $1$ is removed.
This is precisely what we prove in the following theorem.
\begin{thm}[H\"older regularity of the two dimensional dual semigeostrophic equations] 
\label{GSthm} Let $\rho^0$ be a probability measure on $\T^2$.
Suppose that 
that $\lambda \leq \rho^0 \leq \Lambda$ in $\T^2$ for positive constants $\lambda$ and $\Lambda$. Let $(\rho_t, P_t^{\ast})$ solve (\ref{GS:dual})-(\ref{SGbdr}).
Let $P_t$ be the Legendre transform of $P_t^{\ast}$. Then, there 
exist $\alpha=\alpha(\lambda, \Lambda) \in (0,1)$ and $C=C(\lambda,\Lambda)>0$ such that  
$$\|\p_t P_t^{\ast}\|_{L^{\infty}((0,\infty), C^{\alpha}(\T^2))} + \|\p_t P_t\|_{L^{\infty}((0,\infty), C^{\alpha}(\T^2))} \leq C.$$
\end{thm}
We will prove Theorem \ref{GSthm} in Section \ref{SG_sec}, using Theorem \ref{Holder_int_thm} and following the strategy in \cite{Loe}.

Let us briefly explain how to prove the H\"older continuity of $\p_t P_t^{\ast}$ and $\p_t P_t$ in (\ref{GS:dual})-(\ref{SGbdr}).
To simplify the presentation, we assume all functions involved are smooth but the estimates we wish to establish will 
depend only on $\lambda$ and $\Lambda$.
Since $\div U_t=0$, the $L^{\infty}(\T^2)$ norm of $\rho_t$ is preserved in time; see also \cite[Proposition 5.2]{BB} and \cite[Lemma 9.1]{Loe}.
Thus, for all $t\geq 0$, we have $\lambda \leq \rho_t \leq \Lambda$ in $\T^2$. 
Differentiating both sides of $\det D^2 P_t^{\ast} = \rho_t$ with respect to $t$, and using the first and second equations of (\ref{GS:dual}), 
we find that $\p_t P_t^{\ast}$ solves the linearized Monge-Amp\`ere equation
\begin{equation}
\label{SGLMA}
\nabla \cdot (M_{P_t^{\ast}}(D^2 P_t^{\ast}) \nabla (\partial_t P_t^{\ast})) = \p_t \rho_t= \div (-\rho_t U_t):= \div F_t,
\end{equation}
where $M_{P_t^{\ast}}(D^2 P_t^{\ast})$ represents the matrix of cofactors of the Hessian matrix $D^2 P_t^{\ast}$.

With the bounds $\lambda\leq \rho_t\leq\Lambda$ on $\rho_t$, (\ref{SGLMA}) is a degenerate elliptic equation because we only know that the coefficient
matrix $M=M_{P_t^{\ast}}(D^2 P_t^{\ast})$ in (\ref{SGLMA}) is positive definite (due to the convexity of $P_t^{\ast}$) and satisfies
$$\lambda\leq \det M= \det D^2 P_t^{\ast}\leq \Lambda.$$

Moreover, we can bound $F_t$ in $L^{\infty}(\T^2)$ and $\p_t P_t^{\ast}$ in $L^{2}(\T^2)$, uniformly in $t$; see Theorem \ref{thm_collect}(i, ii, iii).
The H\"older regularity of $\p_t P_t^{\ast}$ hence relies on the  H\"older regularity of solutions to equation of the type (\ref{SGLMA}) 
given the
$L^p$ bounds on the solutions, where $F_t$ is a bounded vector field.

At this point, Loeper's approach and assumption on the initial potential density $\rho^0$ and ours differ.

The key tools used by Loeper \cite{Loe} are global and local maximum principles for solutions of degenerate elliptic equations proved by Murthy and Stampacchia 
\cite{MuSt} and Trudinger \cite{Tr}, and a Harnack 
inequality of Caffarelli and Guti\'errez \cite{CG97} for solutions of the homogeneous linearized Monge-Amp\`ere equation; see Theorem \ref{CGthm}. {\it These results hold in all 
dimensions $n\geq 2$.} The results of Murthy-Stampacchia and Trudinger, that we will 
recall in Theorems \ref{2corlocal} and \ref{2main4},  require
the high integrability of the coefficient matrix of the degenerate elliptic equations. In application to the  dual semigeostrophic equations 
 (\ref{GS:dual})-(\ref{SGbdr}), this  high integrability translates to the high integrability of the matrix $M_{P_t^{\ast}}(D^2 P_t^{\ast})$ in (\ref{SGLMA}), or equivalently, to the high
 integrability of $D^2 P_t^{\ast}$. 
 In view of Caffarelli's $W^{2,p}$ estimates for the 
 Monge-Amp\`ere equation \cite{C} and Wang's counterexamples \cite{W}, the last point forces the closeness of the density $\rho^0$ to 1. This is exactly the assumption on $\rho^0$ in \cite[Theorems 2.2, 2.3 and 9.2]{Loe}.

Our main tool in proving Theorem \ref{GSthm} is
the H\"older estimate
in Theorem \ref{Holder_int_thm} for the inhomogeneous linearized Monge-Amp\`ere equation of the type (\ref{SGLMA}) in two dimensions, {\it without relying on $\lambda$ and $\Lambda$ being close. }This is the topic of the next section.

\subsection{H\"older estimates for inhomogeneous linearized Monge-Amp\`ere equation}
Let $\Omega \subset \R^n$ ($n\geq 2$) be a bounded convex set with nonempty interior and let $\varphi \in C^2(\Omega)$ be a convex function such that
\begin{equation}
\label{MAbound}
\lambda\leq \det D^2\varphi \leq \Lambda~\text{in}~\Omega
\end{equation}
for some positive constants $\lambda$ and $\Lambda$.

Let 
$\Phi= (\Phi^{ij})_{1\leq i, j\leq n}=(\det D^2\varphi)(D^2\varphi)^{-1}$ denote the cofactor matrix of 
the Hessian  matrix $$D^2 \varphi=\left(\varphi_{ij}\right)_{1\leq i, j\leq n}\equiv\left(\frac{\p^2\varphi}{\p x_i\p x_j}\right)_{1\leq i, j\leq n}.$$ 
Note that, in terms of the notation of the previous Section \ref{SG_sec1}, we have $\Phi= M_{\varphi}(D^2\varphi)$.

We are interested in obtaining 
interior H\"older
estimates
for solutions to the inhomogeneous linearized Monge-Amp\`ere equation
\begin{equation}
\sum_{i, j=1}^n\Phi^{ij} u_{ij}=\div F
 \label{main_eq}
\end{equation}
in terms of $L^p$ bounds on the solutions  where $F:\Omega\rightarrow \R^n$ is a bounded vector field. 
Our motivation comes from the regularity of the semigeostrophic equations \cite{ACDF12, BB, CG, F} as mentioned in Section \ref{SG_sec1}. 

Since the matrix $\Phi$ is divergence free; that is, 
$\displaystyle\sum_{j=1}^n \p_j \Phi^{ij}=0~\text{for all } i=1, \cdots n, $
the equation (\ref{main_eq}) can also be written in the divergence form
\begin{equation}
\label{main_eq_div}
\sum_{i, j=1}^n\p_j (\Phi^{ij} u_i)\equiv \nabla\cdot (\Phi\nabla u)=\div F.
\end{equation}

When $F\equiv 0$, interior H\"older estimates for solutions to (\ref{main_eq}), under the condition (\ref{MAbound}) on the Monge-Amp\`ere measure of $\varphi$, were 
established by Caffarelli and Guti\'errez in their fundamental work \cite{CG97}. It is worth mentioning that one of the motivations of the work \cite{CG97} was Lagrangian models
of atmospheric and oceanic flows, including the dual semigeostrophic equations.

When $F\not\equiv 0$,
we are able to obtain in this paper the H\"older
estimates
for solutions to (\ref{main_eq}) in two dimensions; see Theorem \ref{Holder_int_thm}. The important point to note here is that our H\"older exponent depends
only on the bounds $\lambda$ and $\Lambda$ of the Monge-Amp\`ere measure of $\varphi$.

Besides its application to the semigeostrophic equations, Theorem \ref{Holder_int_thm} also applies to 
the H\"older regularity of 
the polar 
factorization for time dependent maps in two dimensions with densities bounded away from zero and infinity, improving previous results by Loeper \cite{Loe}; see 
Section \ref{app_sec}.

To state our estimates for (\ref{main_eq}), we recall the notion of sections of a convex function $\varphi \in C^1(\Omega)$.
Given $x \in \Omega$ and $h >0$, the Monge-Amp\`ere section of $\varphi$ centered at $x$ with 
height $h$ is defined by
$$
S_\varphi(x,h):= \{y \in \Omega: \varphi(y) < \varphi(x) +  \nabla \varphi(x)\cdot (y-x) + h\}.
$$

Our main H\"older estimate is contained in the following theorem.
\begin{thm} [Interior H\"older estimate for the inhomogeneous linearized Monge-Amp\`ere equation in two dimensions]
\label{Holder_int_thm}
Assume $n=2$. Let $\varphi \in C^2(\Omega)$ be a convex function satisfying (\ref{MAbound}).
Let $F:\Omega\rightarrow \R^n$ is a bounded vector field. Given a
section $S_\varphi(x_0, 4h_0) \subset \subset \Omega$.
Let $p \in (1, \infty)$.
There exist a universal constant $\gamma >0$ depending only on $\lambda$ and $\Lambda$ 
and a constant $C >0$, depending only on $p$, $\lambda$, $\Lambda, h_0$ and $\diam(\Omega)$ with the following property.
For every solution $u$ to 
$$\Phi^{ij} u_{ij}=\emph{div} F$$ in $S_{\varphi}(x_0, 4h_0)$, and for all $x\in S_{\varphi}(x_0, h_0)$,  we have the H\"older estimate:
 \begin{equation*}
  |u(x)-u(x_0)|\leq  C(p,\lambda,\Lambda, \emph{diam}(\Omega), h_0)  \left( 
 \|F \|_{L^{\infty}(S_\varphi(x_0, 2h_0))} + \|u\|_{L^{p}(S_\varphi(x_0, 2h_0))} \right)|x-x_0|^{\gamma}.
 \end{equation*}
\end{thm}
We will prove Theorem \ref{Holder_int_thm} in Section \ref{LMA_est_proof_sec}.
Our main technical tools, in addition to Caffarelli-Guti\'errez's Harnack inequality for solutions to the homogeneous linearized Monge-Amp\`ere equation in Theorem \ref{CGthm}, are 
new $L^{\infty}$ interior and global estimates for solutions to the inhomogeneous linearized Monge-Amp\`ere equation (\ref{main_eq}) in Theorems 
\ref{gl_thm} and \ref{int_thm}. 

Caffarelli and Guti\'errez \cite{CG97} proved Theorem \ref{CGthm} by using basically the non-divergence form of (\ref{main_eq});
while we prove Theorems \ref{gl_thm} and \ref{int_thm} by exploiting the divergence form character of (\ref{main_eq}). They are related to fine properties
of Green's function $G_{\varphi}$
of the degenerate operator $-\p_i (\Phi^{ij}\p_j)$.
The crucial observation here (see also \cite{L, L2})
is that Green's function $G_{\varphi}$ has, in all dimensions, the same integrability as that of the Laplace operator $\displaystyle \Delta=\sum_{i=1}^n \p_{ii}$ which corresponds to the 
case $\varphi(x)=\frac{|x|^2}{2}$. On the other hand, in two dimensions, the gradient of $G_{\varphi}$ has almost integrability as that of the Laplace operator. We do not know whether the last fact is true or not in higher dimensions.
Thus, it is an open question if the H\"older estimate in Theorem \ref{Holder_int_thm} holds for dimensions $n\geq 3$.

The rest of the paper is organized as follows. In Section \ref{Est_sec},
we provide key global and local estimates in Theorem \ref{gl_thm} and \ref{int_thm} for the inhomogeneous linearized Monge-Amp\`ere equation
and discuss related results by Murthy-Stampacchia and Trudinger.
In Section \ref{sec:prelim}, we recall several basics of the Monge-Amp\`ere equation and its linearization. 
We present the proof of Theorem \ref{GSthm} in Section \ref{SG_sec}.
We prove Theorems \ref{Holder_int_thm}, \ref{gl_thm} and \ref{int_thm}
in Section \ref{LMA_est_proof_sec}.
In Section \ref{app_sec}, we apply Theorem \ref{Holder_int_thm} to 
the regularity of polar factorization of time dependent maps in two dimensions.
The proofs of technical results concerning Green's function that we use in the proofs of Theorems \ref{gl_thm} and \ref{int_thm}  
are presented in Section \ref{auxi_sec}. The proofs of rescaling properties of the Monge-Amp\`ere equation and its linearization will be given in the final section, 
Section \ref{res_proof}.
\section{Estimates for linearized Monge-Amp\`ere equations and related results}
\label{Est_sec}
In this section, we state key global and local estimates for solutions to the inhomogeneous linearized Monge-Amp\`ere equation $\Phi^{ij} u_{ij}=\div F$
and discuss related results by Murthy-Stampacchia and Trudinger regarding solutions to degenerate elliptic equations.
\subsection{Estimates for the equation $\Phi^{ij} u_{ij}=\div F$} 
Our key estimates are the following theorems.

\begin{thm}[Global estimate for solutions to the Dirichlet problem in two dimensions] Assume $n=2$. Let $\varphi \in C^2(\Omega)$ be a convex function satisfying (\ref{MAbound}).
Let $F:\Omega\rightarrow \R^n$ is a bounded vector field.
\label{gl_thm} There exist a universal constant $\delta >0$ depending only on $\lambda$ and $\Lambda$ such that for every 
section $S_\varphi(x_0, h)$ with $S_\varphi(x_0, 2h_0) \subset \subset \Omega$ for $h_0\geq h$ and every solution $u$ to 
\begin{equation}\label{u:PDE:S}
\left\{\begin{array}{rl}
\Phi^{ij} u_{ij} & =\emph{div }F \quad \mbox{ in}\quad S_\varphi(x_0, h),\\
 u &=0 \qquad \quad \,\mbox{on}\quad \partial S_\varphi(x_0, h),
\end{array}\right.
\end{equation}
 we have
 \begin{equation}\label{gl_ineq}
  \sup_{S_\varphi(x_0, h)} |u|\leq C(\lambda,\Lambda,\emph{diam}(\Omega), h_0) \|F\|_{L^{\infty}(S_\varphi(x_0, h))} h^{\delta}.
 \end{equation}
\end{thm}

\begin{thm}[Interior estimate for the inhomogeneous linearized Monge-Amp\`ere equation in two dimensions]
\label{int_thm}
Assume $n=2$. Let $\varphi \in C^2(\Omega)$ be a convex function satisfying (\ref{MAbound}).
Let $F:\Omega\rightarrow \R^n$ is a bounded vector field.
Given $p \in (1, \infty)$, there exists a constant $C >0$, depending only on $p$, $\lambda$, $\Lambda$ and $\emph{diam}(\Omega)$ with the following property:
Every solution 
$u$ of $$\Phi^{ij} u_{ij}=\div F$$ in a 
section $S_\varphi(x_0, h)$ with $S_\varphi(x_0, 2h) \subset \subset \Omega$ satisfies
 \begin{equation}
 \label{main_int_ineq}
  \sup_{S_\varphi(x_0, h/2)} |u|\leq C(p, \lambda,\Lambda,\emph{diam}(\Omega)) \left(\|F\|_{L^{\infty}(S_\varphi(x_0, h))} + h^{-1/p}\|u\|_{L^{p}(S_\varphi(x_0, h))}\right).
 \end{equation}
\end{thm}
We will prove Theorems \ref{gl_thm} and \ref{int_thm}
in Section \ref{LMA_est_proof_sec}.

Given Theorems \ref{gl_thm} and \ref{int_thm}, we can easily prove Theorem \ref{Holder_int_thm} by combining them with Caffarelli-Guti\'errez's Harnack inequality \cite[Theorem 5]{CG97} for
the homogeneous linearized Monge-Amp\`ere equation.  For completeness, we recall their result here.
\begin{thm} [Caffarelli-Guti\'errez's Harnack inequality for the linearized Monge-Amp\`ere equation]
 \label{CGthm}
 Assume $n\geq2$. Let $\varphi \in C^2(\Omega)$ be a convex function satisfying (\ref{MAbound}).
 Let
$u\in W^{2, n}_{\text{loc}}(\Omega)$ be a nonnegative solution of the homogeneous linearized Monge-Amp\`ere equation $$\Phi^{ij} u_{ij}=0$$ in a section $S_\varphi (x_0, 2h)\subset\subset \Omega$. Then
\begin{equation}\sup_{S_{\varphi}(x_0, h)} u\leq C(n, \lambda, \Lambda) \inf_{S_{\varphi}(x_0, h)} u.
 \label{HI2}
\end{equation}
\end{thm}

The main technical tool in the proof of Theorem \ref{gl_thm} is the $L^{1+\kappa}$ estimate $(\kappa>0)$ stated in Proposition \ref{G2_thm} for Green's function associated to the operator
$-\p_j(\Phi^{ij}\p_i)=-\Phi^{ij}\p_{ij}$. 
 We will prove Theorem \ref{int_thm} using the Moser iteration. The main technical tool is the Monge-Amp\`ere Sobolev inequality stated in Proposition \ref{sob_ineq}. We state these 
results in Section \ref{tools_sec}.

\subsection{Integrability of Green's function and its gradient and Monge-Amp\`ere Sobolev inequality}
\label{tools_sec}
Let $\Omega \subset \R^2$ be a bounded convex set with nonempty interior and let $\varphi \in C^2(\Omega)$ be a convex function satisfying (\ref{MAbound}).
Let $g_S(x, y)$ be Green's function of the divergence form operator $\displaystyle L_\varphi:=-\sum_{i,j=1}^2\p_j(\Phi^{ij}\p_i)$ on 
the section $S:= S_\varphi(x_0, h)\subset\subset\Omega$; that is, for 
each $y\in S$, $g_S(\cdot, y)$ is a positive solution of
\begin{equation}\label{Green:S}
\left\{\begin{array}{rl}
L_\varphi g_S(\cdot, y)&=\delta_y \,\, \quad \mbox{ in}\quad S,\\
 g_S(\cdot, y) &=0 \qquad \mbox{on}\quad \partial S.
\end{array}\right.
\end{equation}
Here $\delta_y$ is the Dirac measure centered at $y$.
Due to the divergence free property of $\Phi$, we will also use interchangeably $L_\varphi=-\Phi^{ij}\p_{ij}$ for simplicity.
 The main technical tool in the proof of Theorem \ref{gl_thm} is the 
 following global $L^{1+ \kappa}$ estimates for $\nabla g_S$.
 \begin{prop} [$L^{1+ \kappa}$ estimates for gradient of Green's function]
 \label{G2_thm} 
 Assume $n=2$. Let $\varphi \in C^2(\Omega)$ be a convex function satisfying (\ref{MAbound}). Assume that $S_\varphi(x_0, 2h_0)\subset\subset\Omega$.
 Let $g_S(x, y)$ be Green's function of the operator $L_\varphi:=-\Phi^{ij}\p_{ij}$ on $S:= S_\varphi(x_0, h)$ where $h\leq h_0$, as in (\ref{Green:S}).
 There exist universal constants $\kappa, \kappa_1 > 0$ depending only on $\lambda$ and $\Lambda$ such that for every $y\in S$,
 we have
 $$
 \left(\int_S |\nabla_x g_S (x, y)|^{1+\kappa} dx\right)^{\frac{1}{1+\kappa}} \leq C(\lambda,\Lambda, {\emph{diam}(\Omega)}, h_0) h^{\kappa_1}.
 $$
 \end{prop} 
 \begin{rem}
 Let $\e_\ast=\e_\ast(\lambda,\Lambda)>0$ be the universal constant in De Philippis-Figalli-Savin and Schmidt's $W^{2, 1+\e}$ estimate for the Monge-Amp\`ere 
 equation (\ref{MAbound}); see \cite{DPFS, Fi, Sch} and (\ref{w21est}).
 Then we can choose $\kappa$ and $\kappa_1$ in Proposition \ref{G2_thm} as follows:
 $$\kappa = \frac{\e}{2+\e}, \kappa_1= \frac{\e_\ast-\e}{2(1+\e_\ast)(1+\e)}$$
 where $\e$ is any fixed number in the interval $(0,\e_\ast)$. In the case of $\varphi(x)=|x|^2/2$, $L_{\varphi}=-\Delta$, we have $\e_\ast =\infty$. Thus, in this case, $\kappa$ can be chosen to be any positive 
 number less than $1$, which is optimal.
 \end{rem}
 The main technical tool in the proof of Theorem \ref{int_thm} is the following Monge-Amp\`ere Sobolev inequality;
 it is a two dimensional counterpart of the higher dimensional result in
\cite[Theorem 3.1]{TiW08}.
\begin{prop}[Monge-Amp\`ere Sobolev inequality]
\label{sob_ineq} 
Assume $n=2$. Let $\varphi \in C^2(\Omega)$ be a convex function satisfying (\ref{MAbound}).
Suppose that $S_\varphi (x_0, 2)\subset\subset\Omega$ and $B_1(0)\subset S_\varphi (x_0, 1)\subset B_2(0)$. Then, for every $p\in (2,\infty)$ there exists a constant $K>0$, depending only on $\lambda,\Lambda$ and $p$, such that
\begin{equation}\label{Sob:Phi}
 \left(\int_{S_\varphi(x_0, 1)} |w|^{p}dx\right)^{1/p}\leq K \left(\int_{S_\varphi(x_0, 1)} \Phi^{ij} w_i w_jdx\right)^{1/2} \quad \text{for all }w\in C^1_0(S_\varphi(x_0, 1)).
\end{equation}
\end{prop}
The proofs of Propositions \ref{G2_thm} and \ref{sob_ineq} are based on 
 the following high integrability of $g_S$ when $n=2$ whose proof is based on \cite{L}.
 \begin{prop}[High integrability of Green's function]
 \label{G1_thm} 
 Assume $n=2$. Let $\varphi \in C^2(\Omega)$ be a convex function satisfying (\ref{MAbound}). 
  Let $S:= S_{\varphi}(x_0, h)$ where $S_\varphi(x_0, 2h)\subset\subset\Omega$. 
 Let $g_S(x, y)$ be Green's function of 
 the operator $L_\varphi:=-\Phi^{ij}\p_{ij}$ on $S$, as in (\ref{Green:S}).
Then, for any $p\in (1,\infty)$, we have
 $$
 \int_S g_S^p(x, y) dx \leq C(\lambda,\Lambda, p) h \quad \text{for all }y \in S.
 $$
\end{prop}
We will prove  Propositions \ref{G2_thm}, \ref{sob_ineq} and \ref{G1_thm} in Section \ref{auxi_sec}.
\subsection{Related results by Murthy-Stampacchia and Trudinger}
\label{MST_sect}
Since the matrix $\Phi=(\Phi^{ij})$ in our Theorems \ref{gl_thm} and \ref{int_thm} are divergence free,
the equation
$$\Phi^{ij} u_{ij}=\div F$$
can be written in the divergence form
$$\div (\Phi \nabla u)=\div F.$$
In this section, we discuss related results by Murthy-Stampacchia \cite{MuSt} and Trudinger \cite{Tr} 
concerning the maximum principle, local and global estimates, local and global regularity for degenerate elliptic equations in the divergence form
\begin{equation}
\label{2divforme}
\div (M(x)\nabla u(x))= \div V(x)~\text{in}~\Omega\subset\R^n.
\end{equation}
where $M= (M_{ij})_{1\leq i, j\leq n}$ is  nonnegative symmetric matrix, and
$V$ is a bounded vector field in $\R^n$. 

Without any special structure on the matrix M, it is difficult to obtain the $L^{\infty}$ bound on the solution $u$ to (\ref{2divforme})
in terms of the $L^{\infty}$ bound on the vector field $V$ for equation (\ref{2divforme}) with  Dirichlet boundary data, or in 
terms of the $L^{\infty}$ bound on the vector field $V$ and an integral bound on the solution $u$ in a larger domain. To the best of our knowledge,
some of the strongest results in this generality are due to Murthy-Stampacchia \cite{MuSt} and Trudinger \cite{Tr}. To obtain these results, they require 
high integrability of the matrix $M$ and its inverse. That is, the usual strict ellipticity condition in the classical De Giorgi-Nash-Moser theory (see, for example, Chapter 8 in Gilbarg-Trudinger \cite{GT})
$$
\lambda |\xi|^2 \leq M_{ij}\xi_i\xi_j \leq \Lambda |\xi|^2 ~\text{for some positive constants }\lambda~\text{and }\Lambda, \text{ and for all } \xi \in \R^n,
$$
is replaced by the following condition:
$$\lambda_{M,1}^{-1},~\lambda_{M, 2} \in L^p_{loc}(\Omega)  \text{ for some } p>n,
$$
where $\lambda_{M, 1}(x)$ and $\lambda_{M,2}(x)$ are the smallest and largest eigenvalues of $M(x)$.

We denote by $S_n^+$ the set of $n\times n$ nonnegative symmetric matrices. For reader's convenience, we state the following well known results. 
\begin{thm}
(\emph{Bound for Dirichlet boundary data}; \emph{see} \cite[Chapter 7]{MuSt} \emph{and} \cite[Theorem 4.2]{Tr})
\label{2corlocal}
Let $M=(M_{ij})_{1\leq i, j\leq n}: \Omega \to S_n^+$ be such that $\lambda_{M,1}^{-1}$ is in $L^p_{\text{loc}}(\Omega; S_n^+)$ for some $p>n$. Let $V$ be in $L^{\infty}(\Omega;\R^n)$. 
If $u$ is a solution of (\ref{2divforme}) in $B_R(y) \subset\subset \Omega $ and $u= 0$ on $\partial B_R(y)$, then
$$
\sup_{B_R(y)} |u| \leq C(n, p)\|\lambda_{M, 1}^{-1}\|_{L^p(B_R(y))}\|V\|_{L^{\infty}(B_R(y))} R^{1-\frac{n}{p}}.
$$
\end{thm}

\begin{thm}
(\emph{Bound without boundary data}; \emph{see} \cite[Chapter 8]{MuSt} \emph{and} \cite[Corollary 5.4]{Tr})
\label{2main4}
Let $M= (M_{ij})_{1\leq i, j\leq n}: \Omega \to S_n^+$ be such that $\lambda_{M,2}, \lambda_{M, 1}^{-1}$ are both in $L^p_{loc}(\Omega)$, with $p>n$. Let $V$ be 
in $L^{\infty}(\Omega; \R^n)$. Let $u$ be a solution of (\ref{2divforme})
in $\Omega$. 
Then we have for any ball $B_{2R}(y) \subset \subset \Omega$ and $q>0$
$$
\sup_{B_R(y)} |u| \leq C(\|u\|_{L^{q}(B_{2R}(y))} +  \|V\|_{L^{\infty}(B_{2R}(y))})
$$
where $C$ depends on $n, R, q, p, \|\lambda_{M, 2}\|_{L^p(B_{2R}(y))}$ and $\| \lambda_{M, 1}^{-1}\|_{L^p(B_{2R}(y))}$. 
\end{thm} 

In our Theorems \ref{gl_thm} and \ref{int_thm} in two dimensions, the matrix $\Phi$ belongs to $L^{1+\e_\ast}_{\text{loc}}(\Omega)$,
by De Philippis-Figalli-Savin and Schmidt's $W^{2, 1+\e}$ estimates for the Monge-Amp\`ere equation \cite{DPFS, Sch}. 
Thus, the smallest and largest eigenvalues $\lambda_{\Phi, 1}$ and $\lambda_{\Phi, 2}$ of $\Phi$ satisfies $\lambda_{\Phi, 1}^{-1},\lambda_{\Phi, 2}\in L^{1+\e_\ast}_{\text{loc}}(\Omega)$.
The exponent $\e_\ast = \e_\ast(\lambda,\Lambda)>0$ 
is small and can be taken to be arbitrary close to $0$  when the ratio $\Lambda/\lambda$ is large, by Wang's examples \cite{W}. 
In particular, when $\Lambda/\lambda$ is large, and when $M=\Phi$,
the assumptions in Theorems \ref{2corlocal} and \ref{2main4} on the eigenvalues of $M$ are not satisfied.

On the other hand, in any dimension, when we impose either the continuity or closeness to a positive constant of $\det D^2\varphi$, then 
by Caffarelli's $W^{2,p}$ estimates for the Monge-Amp\`ere equation \cite{C}, $\lambda_{\Phi, 1}^{-1}$ and $\lambda_{\Phi, 2}$ belong to $L^p_{\text{loc}}(\Omega)$
for any $p\in (1,\infty)$. 
Thus, we can apply
Theorems \ref{2corlocal} and \ref{2main4} to (\ref{main_eq}). This is what Loeper used in his proofs of the 
 H\"older regularity of the polar factorization for time-dependent maps and the semigeostrophic equations in \cite[Theorems 2.2, 2.3 and 9.2]{Loe}.

\section{Preliminaries on the Monge-Amp\`ere equation and its linearization}\label{sec:prelim}

Throughout this section we fix a bounded convex set with nonempty interior $\Omega \subset \R^n$ and assume that $\varphi \in C^2(\Omega)$ is a strictly convex function satisfying
\begin{equation}\label{mu:1}
\lambda \leq \det D^2 \varphi \leq \Lambda \quad \text{in } \Omega,
\end{equation}
for some $0 < \lambda \leq \Lambda$. The results in this section hold for all dimensions $n\geq 2$.
\subsection{Basics of the Monge-Amp\`ere equation} We recall in this section some well-known results on the Monge-Amp\`ere equation that we will use in later sections of the paper.
\subsection*{Universal constants} Constants depending only on $\lambda$ and $\Lambda$ in \eqref{mu:1} as well as on dimension $n$ will be called \emph{universal} constants. 

\subsection*{Monge-Amp\`ere sections} Given $x \in \Omega$ and $h >0$, the Monge-Amp\`ere section of $\varphi$ centered at $x$ and with height $h$ is defined as
$$
S_\varphi(x,h):= \{y \in \Omega: \varphi(y) < \varphi(x) + \nabla \varphi(x)\cdot (y-x) + h\}.
$$
A section $S_\varphi(x,h)$ is said to be \emph{normalized} if it satisfies the following inclusions
$$
B_1(0) \subset S_\varphi(x, h) \subset B_n(0),
$$
where $B_r(0)$ denotes the $n$-dimensional ball centered at $0$ and with radius $r >0$. Recall that, by John's lemma, every open bounded convex set with non-empty interior can be normalized
by affine transformations.

\subsection*{Volume estimates for sections} There exists a universal constant $C(n, \lambda, \Lambda) > 0$ such that for every section $S_\varphi(x, h)\subset\subset\Omega$, we have the following volume estimates:
\begin{equation}\label{vol-sec1}
C(n, \lambda, \Lambda)^{-1}h^{n/2} \leq |S_\varphi(x, h)| \leq C(n, \lambda, \Lambda) h^{n/2},
\end{equation}
see \cite[Corollary 3.2.4]{G01}.
\subsection*{$W^{2, 1+\e}$ estimate}
By De Philippis-Figalli-Savin and Schmidt's $W^{2, 1+\e}$ estimates for the Monge-Amp\`ere equation \cite{DPFS, Sch} (see also \cite[Theorem 4.36]{Fi}), there exists $\e_\ast = \e_\ast(n,\lambda,\Lambda)>0$ 
such that
$D^2 \varphi\in L^{1+\e_\ast}_{loc}(\Omega)$.  More precisely, if $S_\varphi(x_0, 1)$ is a normalized section and $S_\varphi(x_0, 2)\subset\subset \Omega$ then
\begin{equation}
\label{w21est}
\|\Delta\varphi\|_{L^{1+\e_\ast}(S_\varphi(x_0, 1))}\leq C(n,\lambda,\Lambda).
\end{equation}
In the following lemma, we estimate the $L^{1+\e_\ast}$ norm of $\Delta\varphi$ and the $C^{\alpha}$ norm of $D\varphi$ on a section $S_{\varphi}(x_0, h)\subset\subset\Omega$. 
They will be applied to Theorem \ref{Holder_int_thm} and Proposition \ref{G2_thm} for $h=h_0$.
\begin{lem}\label{DPFS_lem} 
Let $\varphi \in C^2(\Omega)$ be a convex function satisfying (\ref{MAbound}). 
Let $\e_{\ast}$ be as in (\ref{w21est}).
There exist positive universal constants $\alpha\in (0,1),\alpha_1$ and $\alpha_2$ depending only on $\lambda$,
$\Lambda$ and $n$ such that the following statements hold. If $S_\varphi(x_0, 2h) \subset \subset \Omega$ then

\begin{myindentpar}{1cm}
(i) 
$$\|\Delta\varphi\|_{L^{1+\e_\ast}(S_{\varphi}(x_0, h))} \leq C(\lambda,\Lambda, n,{\emph{diam}}(\Omega)) h^{-\alpha_2}.$$
(ii)
$$|D\varphi(x)-D\varphi(y)|\leq C(\lambda,\Lambda, n,{\emph{diam}}(\Omega))h^{-\alpha_1} |x-y|^{\alpha}~\text{for all }x, y\in S_{\varphi}(x_0, h/2).$$
\end{myindentpar}
\end{lem}
The proof of Lemma \ref{DPFS_lem} will be given in Section \ref{res_proof}. 
\subsection{Rescaling properties for the equation $\Phi^{ij} u_{ij}=\div F$} 
Here we record how the equation \eqref{main_eq} changes with respect to normalization of a section $S_\varphi(x_0, h)\subset\subset\Omega$ of $\varphi$. 

By subtracting $\varphi(x_0) + \nabla\varphi(x_0)\cdot (x-x_0) + h$ from $\varphi$, we may assume 
that $\varphi\mid_{\p S_\varphi(x_0, h)} =0$ and $\varphi$ achieves its minimum $-h$ at $x_0$. By John's lemma, there exists an affine transformation $Tx =A_h x+ b_h$ such that
\begin{equation}\label{normSh}
B_1 (0)\subset T^{-1} (S_\varphi(x_0, h))\subset B_n(0).
\end{equation}
Introduce the following rescaled functions on $T^{-1} (S_\varphi(x_0, h))$:
\begin{equation}
\label{def:tildes}
\left\{
\begin{array}{rl}
\tilde \varphi(x) &:=  (\det A_h)^{-2/n} \varphi(Tx),\\
\tilde u(x) &:= u(Tx),\\
\tilde F(x) &:=(\det A_h)^{\frac{2}{n}} A_h^{-1} F(Tx).
\end{array}
\right.
\end{equation}
Then, we have
\begin{equation}\lambda \leq \det D^2 \tilde \varphi \leq \Lambda~\text{in}~T^{-1} (S_\varphi(x_0, h)),
 \label{detD2tilde}
\end{equation}
with $\tilde \varphi =0$ on $\p T^{-1}(S_\varphi(x_0, h))$ and
$$
B_1(0)\subset \tilde S:=T^{-1}(S_\varphi(x_0, h)) = S_{\tilde \varphi}(\tilde x_0, (\det A_h)^{-2/n} h)\subset B_n(0),
$$
where $\tilde x_0$ is the minimum point of $\tilde \varphi$ in $T^{-1} (S_\varphi(x_0, h))$. 

The following lemma records how the equation \eqref{main_eq} changes with respect to the normalization \eqref{def:tildes} of a section $S_\varphi(x_0, h)\subset\subset\Omega$ of $\varphi$. 
\begin{lem}
\label{res_eq_lem} 
Let $\varphi \in C^2(\Omega)$ be a convex function satisfying (\ref{MAbound}). Let $F:\Omega\rightarrow \R^n$ is a bounded vector field.
Assume that $S_\varphi(x_0, 2h)\subset\subset\Omega$. Under the rescaling \eqref{def:tildes}, the linearized Monge-Amp\`ere equation $$\Phi^{ij} u_{ij}=\div F \text{ in } S_{\varphi}(x_0, h)$$ becomes
\begin{eqnarray}
 \tilde \Phi^{ij}\tilde u_{ij}= \emph{div} \tilde F~\text{in}~\tilde S
 \label{eqSh1}
\end{eqnarray}
with
\begin{equation}
 \|\tilde F\|_{L^{\infty}(\tilde S)}\leq C(\lambda,\Lambda, n,{\emph{diam}}(\Omega))h^{1-\frac{n}{2}}\|F\|_{L^{\infty}(S_\varphi(x_0, h))}
 \label{tFbound}
\end{equation}
and, for any $q > 1$, we have
\begin{equation}
\label{tilde_u_Lq}
  C^{-1}(n,\lambda,\Lambda, q) h^{-\frac{n}{2q}}\|u\|_{L^{q}(S_{\varphi(x_0, h)})}\leq \|\tilde u\|_{L^{q}(\tilde S)} \leq C(n,\lambda,\Lambda, q)
  h^{-\frac{n}{2q}}\|u\|_{L^{q}(S_{\varphi(x_0, h)})}.
\end{equation}
\end{lem}

The proof of Lemma \ref{res_eq_lem} will be given in Section \ref{res_proof}.

\section{Proof of Theorem \ref{GSthm}}
\label{SG_sec}
In this section, we prove Theorem \ref{GSthm}.
We first collect some regularity properties of weak solutions to (\ref{GS:dual})-(\ref{SGbdr}) from previous work by Benamou-Brenier \cite{BB}, Cullen-Feldman \cite{CF}, 
Loeper \cite{Loe}; see also the lecture notes by 
Figalli \cite{F}.
\begin{thm} (\cite{BB, CF, F, Loe})
\label{thm_collect} Let $\rho^0$ be a probability measure on $\T^2$.
Suppose that 
that $ \lambda \leq \rho^0 \leq \Lambda$ in $\T^2$ for positive constants $\lambda$ and $\Lambda$. Let
$(\rho_t, P_t^{\ast})$ solve (\ref{GS:dual})-(\ref{SGbdr}).
Let $P_t$ be the Legendre transform of $P_t^{\ast}$. Then:
\begin{myindentpar}{1cm}
 (i) $\lambda \leq \rho_t\leq \Lambda$ in $\T^2$ for all $t\geq 0$,\\
 (ii) $\|U_t\|_{L^{\infty}(\T^2)}\leq \frac{\sqrt{2}}{2}$ for all $t\geq 0$,\\
 (iii) There is a positive constant $\kappa=\kappa (\lambda,\Lambda)$ such that for all $t>0$, we have $$\int_{\T^2} \rho_t(x) |\p_t \nabla P_t^{\ast}(x)|^{1+\kappa}dx\leq C(\lambda,\Lambda).$$
 (iv) For all $t>0$, $P_t$ is an Aleksandrov solution to the Monge-Amp\`ere equation $$\rho_t(\nabla P_t)\det D^2 P_t =1 \text{ on } \T^2.$$
\end{myindentpar}
\end{thm}
Combining the previous theorem with the known regularity results for strictly convex Aleksandrov solutions to the Monge-Amp\`ere equation, we have the following theorem.
\begin{thm} (\cite{C1, C2, C3, DPFS, Sch})
\label{Cthm}
Let $\rho^0$ be a probability measure on $\T^2$.
Suppose that 
that $ \lambda \leq \rho^0 \leq \Lambda$ in $\T^2$ for positive constants $\lambda$ and $\Lambda$. Let
$(\rho_t, P_t^{\ast})$ solve (\ref{GS:dual})-(\ref{SGbdr}) with the normalization $\int_{\T^2} P_t^\ast(x) dx=0$.
Let $P_t$ be the Legendre transform of $P_t^{\ast}$. Then:
\begin{myindentpar}{1cm}
(i) There exist universal constants $\beta=\beta(\lambda,\Lambda)\in (0, 1)$ and $C= C(\lambda,\Lambda)>0$
such that
\begin{equation*}
  \|P_t^\ast\|_{C^{1,\beta}(\T^2)} +  \|P_t\|_{C^{1,\beta}(\T^2)}  \leq C.
\end{equation*}
(ii) $P_t, P_t^{\ast}\in  L^{\infty}([0,\infty), W^{2, 1+\e_{\ast}}(\T^2))$ for some $\e_{\ast}>0$ depending only on $\lambda$ and $\Lambda$.
\end{myindentpar}
\end{thm}
\begin{rem}
By \cite[Theorem 4.5]{F}, the positive constants $\kappa$ in Theorem \ref{thm_collect} and $\e_\ast$ in Theorem \ref{Cthm} are related by
$$\kappa=\frac{\e_\ast}{2+\e_\ast}.$$
Moreover, the constant $\e_\ast$ in Theorem \ref{Cthm} can be chosen to be the same $\e_\ast$ in (\ref{w21est}) when $n=2$.
\end{rem}

\begin{proof}[Proof of Theorem \ref{GSthm}] 
By an approximation argument as in \cite{ACDF12, Loe}, we only need to establish the bounds in $L^{\infty}((0,\infty); C^{\alpha}(\T^2))$ for $\p_t P_t^\ast$ and $\p_t P_t$ when the solution $(\rho_t, P_t^\ast)$ is smooth as long as these bounds depend only on $\lambda$ and $\Lambda$. Thus, we can
assume in what follows, $\rho_t, U_t$, $P_t^{\ast}$ and $P_t$ are smooth. 

We will use $C$ to denote a generic positive constant depending only on $\lambda$ and $\Lambda$; its value may change from line to line.

By Theorem \ref{thm_collect}(i),
for all $t\geq 0$, we have 
\begin{equation}
\label{rho_ineq}
\lambda \leq \rho_t \leq \Lambda~\text{in } \T^2.
\end{equation}
Differentiating both sides of $\det D^2 P_t^{\ast} = \rho_t$ with respect to $t$, we find that $\p_t P_t^{\ast}$ solves the linearized Monge-Amp\`ere equation
\begin{equation*}
\nabla \cdot (M_{P_t^{\ast}} \nabla (\partial_t P_t^{\ast})) = \p_t \rho_t
\end{equation*}
where $M_{P_t^{\ast}}$ represents the matrix of cofactors of $D^2 P_t^{\ast}$; that is, $M_{P_t^{\ast}}= (\det D^2 P_t^\ast) (D^2 P_t^\ast)^{-1}$.

Using the first and second equations of (\ref{GS:dual}), we get
\begin{equation}
\label{GS_diff}
 \nabla \cdot (M_{P_t^{\ast}} \nabla (\partial_t P_t^{\ast}))=\div F_t
\end{equation}
where $F_t=-\rho_t U_t$ with the following bound obtained from Theorem \ref{thm_collect}(ii):
\begin{equation}
\label{Finfty}
 \|F_t\|_{L^{\infty}(\T^2)} \leq \frac{\sqrt{2}}{2}\Lambda.
\end{equation}
By Theorem \ref{thm_collect}(i, iii), we have
\begin{equation}
\label{Ptkappa}
\int_{\T^2}  |\p_t \nabla P_t^{\ast}(x)|^{1+\kappa}dx\leq C.
\end{equation}
By subtracting a constant from $P_t^{\ast}$, we can assume that for each $t\in (0,\infty)$, 
\begin{equation}\int_{\T^2} P_t^{\ast}(x)dx =0.
 \label{Pnormal}
\end{equation}
 Thus, we deduce from (\ref{Ptkappa}), (\ref{Pnormal}) and the Sobolev imbedding theorem that
\begin{equation}\|\p_t P_t^{\ast}\|_{L^2(\T^2)}\leq C.
 \label{L2bound}
\end{equation}
From (\ref{Pnormal}) and Caffarelli's global $C^{1,\beta}$ estimates \cite{C3} (see also Theorem \ref{Cthm}), we find 
universal constants $\beta=\beta(\lambda,\Lambda)\in (0, 1)$ and $C= C(\lambda,\Lambda)>0$
such that
\begin{equation}
 \label{C1alpha}
  \|P_t^\ast\|_{C^{1,\beta}(\T^2)} +  \|P_t\|_{C^{1,\beta}(\T^2)}  \leq C.
\end{equation}
Using (\ref{C1alpha}) together with the $\Z^2$-periodicity of $P_t^{\ast}-|x|^2/2$, we can find positive constants $h_0(\lambda,\Lambda)$ and $R_0(\lambda,\Lambda)$ such that
\begin{equation}\T^2\subset S_{P_t^\ast}(x_0, h_0)\subset S_{P_t^\ast}(x_0, 4h_0)\subset B_{R_0}(0)~\text{for all } x_0\in \T^2.
 \label{Sconfine}
\end{equation}
Again, using the $\Z^2$-periodicity of $P_t^{\ast}-|x|^2/2$ and $F_t$, we deduce from (\ref{Finfty}) and (\ref{L2bound}) that
\begin{equation} \|F_t\|_{L^{\infty}(B_{R_0}(0))} + \|\p_t P_t^{\ast}\|_{L^2(B_{R_0}(0))}\leq C(\lambda,\Lambda).
 \label{L2bound2}
\end{equation}
With (\ref{rho_ineq}), (\ref{Sconfine}) and (\ref{L2bound2}) in hand, we can apply 
Theorem \ref{Holder_int_thm} to (\ref{GS_diff}) in each section $S_{P_t^\ast}(x_0, 4h_0)$ with $p=2$ and $\Omega= B_{R_0}(0)$ to conclude that: For some
$\gamma=\gamma(\lambda,\Lambda)\in (0,1)$, we have
\begin{equation}
\label{tPt1}
\|\p_t P_t^{\ast}\|_{L^{\infty}((0,\infty), C^{\gamma}(\T^2))} \leq C(\lambda,
\Lambda).
\end{equation}
For the H\"older regularity of  $\partial_t P_t$, we use 
the
equation 
\begin{equation}\partial_t P_t(x) = - \partial_t P_t^{\ast}(\nabla P_t(x))~\text{for } x\in \R^2
 \label{PPt_eq}
\end{equation}
which follows from differentiating with respect to $t$ the equation
$$P_t(x)+ P_t^{\ast}(\nabla P_t(x))= x\cdot \nabla P_t(x) ~\text{for } x\in \R^2.$$
Combining (\ref{tPt1}) with (\ref{C1alpha}),
we obtain from (\ref{PPt_eq}) the following H\"older estimate for $\p_t P_t$:
$$\|\p_t P_t\|_{L^{\infty}((0,\infty), C^{\gamma\beta}(\T^2))} \leq C(\lambda,
\Lambda).$$
The proof of Theorem \ref{GSthm} is complete by setting $\alpha=\gamma\beta$.
\end{proof}

\section{Proof of Theorems \ref{Holder_int_thm}, \ref{gl_thm}, and \ref{int_thm}}
In this section, we prove Theorems \ref{Holder_int_thm}, \ref{gl_thm}, and \ref{int_thm}.
\begin{proof} [Proof of Theorem \ref{Holder_int_thm}]
The conclusion of the theorem follows from Lemma \ref{DPFS_lem} and the following oscillation estimate: 
For every $h\leq h_0$, we have 
 \begin{equation}
  \osc(u, S_\varphi(x_0, h))\leq  C(p,\lambda,\Lambda, \diam(\Omega), h_0)  \left( 
 \|F \|_{L^{\infty}(S_\varphi(x_0, 2h_0))} + \|u\|_{L^{p}(S_\varphi(x_0, 2h_0))} \right)h^{\gamma_0}
 \label{key_osc}
 \end{equation}
where $\gamma_0\in (0, 1)$ depends only on $\lambda$ and $\Lambda$. Here and what follows, we use the following notation for a function $f$ defined on a set $E$:
$$
\osc(f, E):= \sup\limits_{E} f - \inf\limits_{E} f.
$$
Indeed, suppose (\ref{key_osc}) is established. For each $x\in S_{\varphi}(x_0, h_0)\backslash \{x_0\}$, we can find some $h\in (0, h_0]$ such that $x\in \p S_{\varphi}(x_0, h)$.
By the mean value theorem, we can find some $z$ in the interval from $x_0$ to $x$ such that
$$h=\varphi(x)-\varphi(x_0)-\nabla \varphi(x_0)\cdot (x-x_0)= (\nabla \varphi(z)-\nabla \varphi(x_0))\cdot (x-x_0).$$
The $C^{1,\alpha}$ estimate in Lemma \ref{DPFS_lem} applied to $z, x_0\in S_{\varphi}(x_0, h_0)$ then gives
$$h\leq C(\lambda,\Lambda, \diam(\Omega), h_0)|z-x_0|^{\alpha}|x-x_0|\leq C(\lambda,\Lambda, \diam(\Omega), h_0)|x-x_0|^{1+ \alpha}.$$
Using (\ref{key_osc}), we find that
\begin{eqnarray*}
 |u(x)-u(x_0)|&\leq&  C(p,\lambda,\Lambda, \diam(\Omega), h_0) (\|F \|_{L^{\infty}(S_\varphi(x_0, 2h_0))} + \|u\|_{L^{p}(S_\varphi(x_0, 2h_0))} )h^{\gamma_0}\\
 & \leq& C(p,\lambda,\Lambda, \diam(\Omega), h_0)  \left( 
 \|F \|_{L^{\infty}(S_\varphi(x_0, 2h_0))} + \|u\|_{L^{p}(S_\varphi(x_0, 2h_0))} \right)|x-x_0|^{\gamma}
\end{eqnarray*}
where
$\gamma=\gamma_0(1+\alpha).$ The conclusion of Theorem \ref{Holder_int_thm} follows.

It remains to prove (\ref{key_osc}).
On the section $S_{\varphi}(x_0, h)$ with $h\leq h_0$, we split $u$ as $u= v + w$ where
\begin{equation*}
\left\{\begin{array}{rl}
\Phi^{ij} v_{ij} & =\div F \quad \mbox{ in}\quad S_\varphi(x_0, h),\\
 v &=0 \qquad \quad \,\mbox{on}\quad \partial S_\varphi(x_0, h),
\end{array}\right.
\end{equation*}
and
\begin{equation*}
\left\{\begin{array}{rl}
\Phi^{ij} w_{ij} & = 0\quad\quad \quad \mbox{ in}\quad S_\varphi(x_0, h),\\
 w &=u \qquad \quad \,\mbox{on}\quad \partial S_\varphi(x_0, h).
\end{array}\right.
\end{equation*}
By Theorem \ref{gl_thm} applied to the equation for $v$, we can find a universal constant $\delta=\delta(\lambda,\Lambda)$ such that 
\begin{equation}\label{vSh}
\sup_{S_\varphi(x_0, h)} |v|\leq C(\lambda,\Lambda, \diam(\Omega), h_0) \|F \|_{L^{\infty}(S_\varphi(x_0, h))} h^{\delta}.
\end{equation}
On the other hand, as a consequence of Theorem \ref{CGthm} (see also the Corollary in \cite[p.455]{CG97}), we have
from the homogeneous linearized Monge-Amp\`ere equation for $w$ that
\begin{equation*}
\osc(w, S_\varphi(x_0, h/2)) \leq \beta \osc(w, S_\varphi(x_0, h))
\end{equation*}
for some $\beta=\beta(\lambda, \Lambda) \in (0,1)$.
Therefore,  
\begin{eqnarray}
\label{uvwsplit}
\osc(u, S_\varphi(x_0, h/2)) &\leq& \osc(w, S_\varphi(x_0, h/2)) + \osc(v, S_\varphi(x_0, h/2))\nonumber\\
 &\leq& \beta \osc(w, S_\varphi(x_0, h)) + 2 \norm{v}_{L^\infty(S_\varphi(x_0, h/2))}.
\end{eqnarray}
Using the maximum principle to $\Phi^{ij} w_{ij}=0$, we have
\begin{eqnarray*}
 \osc(w, S_\varphi(x_0, h)) = \osc(w, \p S_\varphi(x_0, h))=\osc(u, \p S_\varphi(x_0, h))\leq \osc(u, S_\varphi(x_0, h)).
\end{eqnarray*}
Together with (\ref{uvwsplit}) and (\ref{vSh}), we find for every $h \leq h_0$ 
$$\osc(u, S_\varphi(x_0, h/2))\leq \beta \osc(u, S_\varphi(x_0, h)) + C(\lambda,\Lambda, \diam(\Omega), h_0) \|F \|_{L^{\infty}(S_\varphi(x_0, h_0))} h^{\delta}.$$
Hence, by a standard argument (see, for example, Han-Lin \cite[Lemma 4.19]{HL}), we have for all $h\leq h_0$
\begin{eqnarray}
\label{osc_est1}
\osc(u, S_\varphi(x_0, h)) &\leq& C\left(\frac{h}{h_0} \right)^{\gamma_0}  \left(  \osc(u, S_\varphi(x_0, h_0)) + 
C_1(\lambda,\Lambda, \diam(\Omega), h_0) \|F \|_{L^{\infty}(S_\varphi(x_0, h_0))} h_0^{\delta} \right) 
 \nonumber\\&\leq&  C \left(\frac{h}{h_0} \right)^{\gamma_0}  \left(2\|u\|_{L^{\infty}(S_\varphi(x_0, h_0))}+ 
C_1\|F \|_{L^{\infty}(S_\varphi(x_0, h_0))} h_0^{\delta} \right) 
\end{eqnarray}
for a structural constant $\gamma_0=\gamma_0(\lambda,\Lambda) \in (0,1)$ and some constant $C_1= C_1(\lambda,\Lambda,\diam(\Omega), h_0)$. \\
By Theorem \ref{int_thm}, we have
$$\|u\|_{L^{\infty}(S_\varphi(x_0, h_0))}\leq C(p, \lambda,\Lambda,\diam(\Omega))
\left(\|F\|_{L^{\infty}(S_\varphi(x_0, 2h_0))} + h_0^{-1/p}\|u\|_{L^{p}(S_\varphi(x_0, 2h_0))}\right).$$
The above estimate combined with (\ref{osc_est1}) gives (\ref{key_osc}). The proof of Theorem \ref{Holder_int_thm} is complete. 
\end{proof}
\label{LMA_est_proof_sec}
\begin{proof}[Proof of Theorem \ref{gl_thm}]
 Let $g_S(x, y)$ be Green's function of the operator $L_\varphi=-\p_j(\Phi^{ij}\p_i)=-\Phi^{ij}\p_{ij}$ on $S= S_\varphi(x_0, h)$, that is, $g_S$ satisfies \eqref{Green:S}. Then, by using that $u$ solves $\Phi^{ij} u_{ij}=\nabla \cdot F$ with $u=0$ on $\partial S$, 
 we get
 $$
 u(y)=-\int_{S}g_S(y, x) \nabla \cdot F (x) dx \quad \text{for all }y \in S.
 $$
 Using symmetry of Green's function and 
 integrating by parts, we obtain for all $y\in S$
 \begin{equation}\label{rep:u:F}
 u(y)=-\int_{S}g_S(x, y) \nabla \cdot F (x) dx = \int_{S} \langle \nabla_x g_S(x, y), F (x) \rangle dx.
 \end{equation}
 It follows that for all $y\in S$, we have
 \begin{eqnarray*}
 |u(y)|\leq \|F\|_{L^{\infty}(S)} \int_{S}| \nabla_x g_S(x, y)| dx \leq \|F\|_{L^{\infty}(S)} \left(\int_{S}| 
 \nabla_x g_S(x, y)|^{1+\kappa} dx\right)^{\frac{1}{1+\kappa}} |S|^{\frac{\kappa}{1+\kappa}}.
 \end{eqnarray*}
 From the $L^{1+\kappa}$-bound for $\nabla g_S$ in Proposition \ref{G2_thm}, we have
$$\left(\int_{S}| \nabla_x g_S(x, y)|^{1+\kappa} dx\right)^{\frac{1}{1+\kappa}}\leq C(\lambda,\Lambda, \diam(\Omega), h_0) h^{\kappa_1}$$
 Thus, by the volume estimates for sections in \eqref{vol-sec1}, we  obtain the asserted $L^{\infty}(S)$ bound  for $u$ from
$$ \|u\|_{L^{\infty}(S)}\leq C(\lambda,\Lambda, \diam(\Omega), h_0) \|F\|_{L^{\infty}(S)} h^{\kappa_1}|S|^{\frac{\kappa}{1+\kappa}}\leq C(\lambda,\Lambda, \diam(\Omega), h_0) \|F\|_{L^{\infty}(S)}
h^{\kappa_1 + \frac{\kappa}{1+\kappa}}.$$
\end{proof}
The rest of this section is devoted to the proof of  Theorem \ref{int_thm} using Moser's iteration.
The key step is to prove the theorem when the section $S_{\varphi}(x_0, h)$ is normalized (that is, when it is comparable to the unit ball) and when we have a high integrability of the solution. This is the content of Proposition
\ref{int_thm_v2}. After this, the theorem easily follows from a rescaling argument.
\begin{prop}\label{int_thm_v2} 
Assume $n=2$. Let $\varphi \in C^2(\Omega)$ be a convex function satisfying (\ref{MAbound}).
Let $F:\Omega\rightarrow \R^n$ is a bounded vector field.
There exist universally large constants $C_0 > 1$ and $p_0 > 2$ depending only on $\lambda$ and $\Lambda$ such 
 that for every solution $u$ of (\ref{main_eq}) in a section $S_\varphi(x_0, 2)\subset\subset\Omega$ with
 $B_1(0)\subset S_\varphi (x_0, 1)\subset B_2(0),$ 
 we have
\begin{equation}
  \sup_{S_\varphi(x_0, 1/2)} |u|\leq C_0 \left(  \|u\|_{L^{p_0}(S_\varphi(x_0, 1))}+ \|F\|_{L^{\infty}(S_\varphi(x_0, 1))}\right).
  \label{int_ineq_v2}
 \end{equation}
\end{prop}
Note that, by the volume estimates (\ref{vol-sec1}), any normalized section $S_{\varphi}(x_0, h)\subset\subset\Omega$
has height $h$ with $c(\lambda,\Lambda)\leq h\leq C(\lambda,\Lambda)$. Our proof of Proposition \ref{int_thm_v2} works for all these $h$. However, to simplify the presentation, we choose to work with $h=1$ in Proposition \ref{int_thm_v2}.

By combining Proposition \ref{int_thm_v2} and Lemma \ref{res_eq_lem} we immediately obtain:

\begin{cor}\label{int_thm_v3}
Assume $n=2$. Let $\varphi \in C^2(\Omega)$ be a convex function satisfying (\ref{MAbound}).
Let $F:\Omega\rightarrow \R^2$ is a bounded vector field.
There exist a universal constant $p_0=p_0(\lambda,\Lambda)$  and a constant $C_1$ depending only on $\lambda,\Lambda$ and $\emph{diam}(\Omega)$ such that for every 
solution $u$ of \eqref{main_eq} in a section 
$S_\varphi(x_0, 2h) \subset \subset \Omega$, we have
\begin{equation}
  \sup_{S_\varphi(x_0, h/2)} |u| \leq  C_1\left(\|F\|_{L^{\infty}(S_\varphi(x_0, h))} + h^{-\frac{n}{2p_0}}
  \|u\|_{L^{p_0}(S_\varphi(x_0, h))}\right).
  \label{int_ineq_v3}
 \end{equation}
\end{cor}
 Now, with Corollary \ref{int_thm_v3}, we are ready to give the proof of Theorem \ref{int_thm}.

\begin{proof}[Proof of Theorem \ref{int_thm}] We show that 
(\ref{main_int_ineq}) follows from (\ref{int_ineq_v3}).
The proof is based on a simple rescaling argument as in the classical proof of the local boundedness of solutions to uniformly elliptic equations 
(see, for example, \cite[Theorem 4.1]{HL}) with Euclidean balls replaced by sections of $\varphi$.

By \cite[Theorem 3.3.10(i)]{G01}, there exist universal constants $c_1 >0$ and $\mu >0$ such that for every $\theta \in (0,1)$ and $y \in S_\varphi(x_0, \theta h)$ we have the inclusion
\begin{equation}\label{inc:Sy:Sx0}
S_\varphi(y, c_1 (1-\theta)^{\mu} h) \subset S_\varphi(x_0, h).
\end{equation}
Then, by applying \eqref{int_ineq_v3} to $u$ on the section $S_\varphi(y, c_1 (1-\theta)^{\mu} h)$ we obtain
\begin{eqnarray*}
|u(y)| &\leq& \sup_{S_\varphi(y, c_1 (1-\theta)^{\mu} h/2)} |u| \\ &\leq&  C_1 (\lambda,\Lambda,\text{diam}(\Omega))\left(\|F\|_{L^{\infty}(S_\varphi(y, c_1 (1-\theta)^{\mu} h))}  +
(1-\theta)^{-\frac{\mu}{p_0}}h^{-\frac{1}{p_0}}\|u\|_{L^{p_0}(S_\varphi(y, c_1 (1-\theta)^{\mu} h))}\right).
 \end{eqnarray*}
Varying $y\in  S_{\varphi}(x_0, \theta h)$ and recalling  \eqref{inc:Sy:Sx0}, we obtain
\begin{equation}\label{mvi_theta}
\norm{u}_{L^\infty(S_\varphi(x_0, \theta h))}  \leq 
C_1 (\lambda,\Lambda,\text{diam}(\Omega))\left(\|F\|_{L^{\infty}(S_\varphi(x_0, h))}  +
(1-\theta)^{-\frac{\mu}{p_0}}h^{-\frac{1}{p_0}}\|u\|_{L^{p_0}(S_\varphi(x_0, h))}\right).
\end{equation}
Now, given $p \in (1, p_0)$ we obtain from \eqref{mvi_theta} the estimate
\begin{equation*}
\norm{u}_{L^\infty(S_\varphi(x_0, \theta h))} \leq
 C_1\left(\|F\|_{L^{\infty}(S_\varphi(x_0, h))}  +
((1-\theta)^{\mu} h)^{-\frac{1}{p_0}}\norm{u}_{L^\infty(S_\varphi(x_0, h))}^{1-\frac{p}{p_0}} \|u\|^{\frac{p}{p_0}}_{L^{p}(S_\varphi(x_0, h))}\right).
\end{equation*}
By Young's inequality with two exponents $p_0/p$ and $p_0/(p_0 - p)$, we have
\begin{multline*}
 C_1((1-\theta)^{\mu} h)^{-\frac{1}{p_0}}\norm{u}_{L^\infty(S_\varphi(x_0, h))}^{1-\frac{p}{p_0}} \|u\|^{\frac{p}{p_0}}_{L^{p}(S_\varphi(x_0, h))}=
 \norm{u}_{L^\infty(S_\varphi(x_0, h))}^{1-\frac{p}{p_0}} C_1((1-\theta)^{\mu} h)^{-\frac{1}{p_0}}
 \|u\|^{\frac{p}{p_0}}_{L^{p}(S_\varphi(x_0, h))}\\ \leq (1-\frac{p}{p_0})\norm{u}_{L^\infty(S_\varphi(x_0, h))}
 + \frac{p}{p_0}C_1^{\frac{p_0}{p}} ((1-\theta)^{\mu} h)^{-\frac{1}{p}}\|u\|_{L^{p}(S_\varphi(x_0, h))}.
\end{multline*}
Hence, for every $\theta \in (0,1)$ we have for a constant $C_2$ depending only on $\lambda,\Lambda, \text{diam}(\Omega)$
\begin{align*}
\norm{u}_{L^\infty(S_\varphi(x_0, \theta h))} \leq (1-\frac{p}{p_0}) \norm{u}_{L^\infty(S_\varphi(x_0, h))} 
 +  C_2 \left(\|F\|_{L^{\infty}(S_\varphi(x_0, h))} +  (1-\theta)^{-\frac{\mu}{p}} h^{-\frac{1}{p}}\|u\|_{L^{p}(S_\varphi(x_0, h))}\right).
\end{align*}
It is now standard (see \cite[Lemma 4.3]{HL}) that 
for every $p \in (1, p_0)$, we get
\begin{align*}
\norm{u}_{L^\infty(S_\varphi(x_0, \theta h))} \leq 
C_3 \left(\|F\|_{L^{\infty}(S_\varphi(x_0, h))} +  (1-\theta)^{-\frac{\mu}{p}}h^{-\frac{1}{p}}\|u\|_{L^{p}(S_\varphi(x_0, h))}\right)
\end{align*}
for a constant $C_3$ depending only on $p, \lambda,\Lambda$ and $\text{diam}(\Omega)$. Theorem \ref{int_thm} follows from the above estimate by setting $\theta=1/2$.
\end{proof}
To complete the proof of Theorem \ref{int_thm}, it remains to prove
Proposition \ref{int_thm_v2}.
\begin{proof}[Proof of Proposition \ref{int_thm_v2}] 
Let $\e=\e_\ast(\lambda,\Lambda)>0$ be the universal constant in De Philippis-Figalli-Savin and Schmidt's $W^{2, 1+\e}$ estimate; see \cite{DPFS, Sch} and (\ref{w21est}).
Then, by the convexity of $\varphi$, we have in two dimensions 
\begin{equation}\|\Phi\|_{L^{1+\e}(S_\varphi(x_0, 1))}= \|D^2 \varphi\|_{L^{1+\e}(S_\varphi(x_0, 1))} \leq \|\Delta\varphi\|_{L^{1+\e}(S_\varphi(x_0, 1))}\leq C(\lambda,\Lambda).
 \label{W21}
\end{equation} 
We prove the proposition for a large constant $C_0$ depending only on $\lambda$ and $\Lambda$ and
$$p_0:=2 \frac{\e+1}{\e}.$$
 By the homogeneity of \eqref{main_eq}, we can assume that
 \begin{equation}\|F\|_{L^{\infty}(S_\varphi(x_0, 1))} + \|u\|_{L^{p_0}(S_\varphi(x_0, 1))}\leq 1.
  \label{Fbound}
 \end{equation}
In order to prove \eqref{int_ineq_v2}, we then need to show that,  for some universal constant $C_0 > 0$, we have
 \begin{equation}
 \label{uboundC}
 \sup_{S_\varphi(x_0, 1/2)} |u|\leq C_0.
 \end{equation}
We will use Moser's iteration to prove the proposition. Given $r \in (0, 1]$, let us put $$S_r:= S_\varphi(x_0, r) ~\text{and } S:= S_1=S_\varphi(x_0, 1).$$
Let $\eta\in C^1_0(S)$ be a cut-off function to be determined later. Let $\beta\geq 0$.  By testing (\ref{main_eq}) against $|u|^{\beta} u \eta^2$ using its divergence form (\ref{main_eq_div}), we get
 \begin{eqnarray}
  \int_S F \cdot \nabla (|u|^{\beta} u \eta^2)dx&=& \int_S \Phi^{ij} u_i (|u|^{\beta} u \eta^2)_jdx\nonumber \\&=&  (\beta+1)\int_S \Phi^{ij} u_i u_j |u|^{\beta}  \eta^2 dx + 
  2\int_S \Phi^{ij} u_i \eta_j \eta |u|^\beta u dx.
  \label{Ftest}
 \end{eqnarray}
 Next, the Cauchy-Schwarz inequality gives
\begin{eqnarray}
 \left|2\int_S \Phi^{ij} u_i \eta_j \eta |u|^\beta u dx\right|&\leq& 2 \left(\int_S\Phi^{ij}u_i u_j |u|^{\beta} \eta^2 dx\right)^{1/2} \left(\int_S\Phi^{ij}\eta_i\eta_j 
 |u|^{\beta+2}dx\right)^{1/2}\nonumber\\
 &\leq& \frac{1}{2}\int_S\Phi^{ij}u_i u_j |u|^{\beta} \eta^2 dx + 2\int_S\Phi^{ij}\eta_i \eta_j |u|^{\beta+2} dx.
 \label{CS_ineq}
\end{eqnarray}
It follows from (\ref{Ftest}) that
 \begin{equation}(\beta +\tfrac{1}{2})\int_S\Phi^{ij}u_i u_j |u|^{\beta} \eta^2 dx - 2\int_S\Phi^{ij}\eta_i \eta_j |u|^{\beta+2}dx\leq  
 \int_S F \cdot \nabla (|u|^{\beta} u \eta^2)dx:=M.
 \label{eq_test1}
 \end{equation}
 We now handle the right hand side $M$ of (\ref{eq_test1}). First,
 using (\ref{Fbound}), we have
\begin{eqnarray*}
  M=\int_S F \cdot \nabla (|u|^{\beta} u \eta^2) dx&=& (\beta + 1)\int_S F \cdot \nabla u |u|^{\beta}  \eta^2 dx +  2\int_S F \cdot \nabla \eta \eta |u|^{\beta} u dx
   \\ &\leq&  (\beta + 1)\int_S |\nabla u| |u|^{\beta}  \eta^2 dx +  2\int_S |\nabla \eta| \eta |u|^{\beta+1} dx
\end{eqnarray*}
Second, using $\det D^2\varphi\geq\lambda$ and the
 following inequality
$\Phi^{ij} v_{i}(x) v_{j}(x)  \geq \frac{\det D^2 \varphi |\nabla v|^2}{\Delta \varphi}
$
whose simple proof can be found in \cite[Lemma 2.1]{CG97}, we deduce that
\begin{equation*}
 M \leq 
  (\beta + 1)\int_S \lambda^{-\frac{1}{2}} \left(\Phi^{ij} u_i u_j \Delta \varphi\right)^{\frac{1}{2}}|u|^{\beta}  \eta^2 dx +  2\int_S \lambda^{-\frac{1}{2}} 
   \left(\Phi^{ij} \eta_i \eta_j \Delta \varphi\right)^{\frac{1}{2}}\eta |u|^{\beta+1}dx.
\end{equation*}
Using the Cauchy-Schwartz inequality, we obtain
\begin{eqnarray*}M &\leq& \frac{\beta+ 1}{4}\int_S\Phi^{ij}u_i u_j |u|^{\beta} \eta^2 dx + 
C(\lambda)(\beta+1) \int_S \Delta\varphi |u|^{\beta} \eta^2 dx\\
&+& C(\lambda)\int_S \Delta\varphi |u|^{\beta} \eta^2 dx + \int_S\Phi^{ij}\eta_i \eta_j |u|^{\beta+2} dx
\\&\leq & \frac{\beta+ 1}{4}\int_S\Phi^{ij}u_i u_j |u|^{\beta} \eta^2 dx + 
C(\lambda)(\beta+1) \int_S \Delta\varphi |u|^{\beta} \eta^2 dx + \int_S\Phi^{ij}\eta_i \eta_j |u|^{\beta+2} dx.
 \end{eqnarray*}
It follows from (\ref{eq_test1}) that
\begin{equation}
 \label{eq_test6}
  \frac{\beta+ 2}{8}\int_S\Phi^{ij}u_i u_j |u|^{\beta} \eta^2 dx\leq 
C(\lambda)(\beta+2) \int_S \Delta\varphi |u|^{\beta} \eta^2 dx + 3\int_S\Phi^{ij}\eta_i \eta_j |u|^{\beta+2} dx.
\end{equation}
Consider the quantity
\begin{eqnarray}
\label{Qdef}
 Q&:=& \int_S \Phi^{ij} (|u|^{\beta/2} u \eta)_i (|u|^{\beta/2} u \eta)_j dx \nonumber\\ &=& \frac{(\beta+2)^2}{4} \int_S\Phi^{ij}u_i u_j |u|^{\beta} \eta^2 dx +
 (\beta +2 )\int_S \Phi^{ij} u_i\eta_j \eta
 |u|^{\beta} u dx + \int_S\Phi^{ij}\eta_i \eta_j |u|^{\beta+2} dx.
\end{eqnarray}
Using (\ref{CS_ineq}), we obtain
\begin{multline*}
 Q\leq  \frac{(\beta+2)^2}{4} \int_S\Phi^{ij}u_i u_j |u|^{\beta} \eta^2 dx + \frac{\beta+2}{4}\int_S\Phi^{ij}u_i u_j |u|^{\beta} \eta^2 dx\\ + 
 (\beta+2)\int_S\Phi^{ij}\eta_i \eta_j |u|^{\beta+2} dx + \int_S\Phi^{ij}\eta_i \eta_j |u|^{\beta+2} dx\\
 \leq (\beta+2)^2 \int_S\Phi^{ij}u_i u_j |u|^{\beta} \eta^2 dx + (\beta +4)\int_S\Phi^{ij}\eta_i \eta_j |u|^{\beta+2} dx.
\end{multline*}
By (\ref{eq_test6}), we have
\begin{eqnarray*}
 \label{eq_test7}
 \frac{24}{\beta +2} \int_S\Phi^{ij}\eta_i \eta_j |u|^{\beta+2} dx + C(\lambda) \int_S \Delta\varphi |u|^{\beta} \eta^2 dx &\geq& \int_S\Phi^{ij}u_i u_j |u|^{\beta}\eta^2 dx
 \\ &\geq &  \left[Q- (\beta+4)\int_S\Phi^{ij}\eta_i  \eta_j |u|^{\beta+2} dx\right]\frac{1}{(\beta+2)^2}. 
\end{eqnarray*}
Therefore, we have 
\begin{equation}
  24(\beta +2) \int_S\Phi^{ij}\eta_i \eta_j |u|^{\beta+2} dx + C(\lambda) (\beta+2)^2 \int_S \Delta\varphi |u|^{\beta} \eta^2 dx\geq Q.
  \label{eq_test3}
\end{equation}
We will bound from above each term on the left hand side of (\ref{eq_test3}).
Using H\"older's inequality and (\ref{W21}), we get
\begin{equation}
\label{Q1ineq}
 \int_S\Phi^{ij}\eta_i \eta_j |u|^{\beta+2} dx\leq \|\nabla \eta\|^2_{L^\infty(S)} \|\Phi\|_{L^{1+\e}(S)} \|u\|^{\beta+2}_{L^{\frac{(\beta+2)(\e+1)}{\e}}(S)}
 \leq C\|\nabla \eta\|^2_{L^\infty(S)} \|u\|^{\beta+2}_{L^{\frac{(\beta+2)(\e+1)}{\e}}(S)}
\end{equation}
and
\begin{multline}
\label{Q2ineq}
 \int_S \Delta\varphi |u|^{\beta} \eta^2 dx\leq C\|\eta\|^2_{L^\infty(S)} \|u\|^{\beta}_{L^{\frac{\beta(\e+1)}{\e}}(S)}\|\Delta\varphi\|_{L^{1+\e}(S)} \\
 \leq C\|\eta\|^2_{L^\infty(S)} |S|^{\frac{2\e}{(\e+1)(\beta+2)}}
 \|u\|^{\beta}_{L^{\frac{(\beta+2)(\e+1)}{\e}}(S)}
 \leq C\|\eta\|^2_{L^\infty(S)}
 \|u\|^{\beta}_{L^{\frac{(\beta+2)(\e+1)}{\e}}(S)}.
\end{multline}

We now apply the Sobolev inequality in Proposition \ref{sob_ineq} to the function $w= |u|^{\beta/2}u \eta$ and the exponent $q= 4 \frac{\e+1}{\e}$. We then have from the definition of $Q$ in (\ref{Qdef}) that
$$Q\geq C\|w\|_{L^q(S)}^2= C\left(\int_S |u|^{2(\beta+2)\frac{\e+1}{\e}} \eta^q dx\right)^{2/q}.$$
Thus, invoking (\ref{Q1ineq}) and (\ref{Q2ineq}), we obtain from (\ref{eq_test3}) the estimate
\begin{equation*}
\left(\int_S |u|^{2(\beta+2)\frac{\e+1}{\e}} \eta^q dx\right)^{2/q} \leq C (\beta +2)\|\nabla \eta\|^2_{L^{\infty}(S)} \|u\|^{\beta+2}_{L^{\frac{(\beta+2)(\e+1)}{\e}}(S)}
+  C(\beta+2)^2\|\eta\|^2_{L^\infty(S)}
 \|u\|^{\beta}_{L^{\frac{(\beta+2)(\e+1)}{\e}}(S)}.
\end{equation*}
Let $\gamma := (\beta+2) \frac{\e+1}{\e}.$ Then
\begin{equation}
 \left(\int_S |u|^{2\gamma}\eta^q dx\right)^{\frac{\beta+2}{2\gamma}}\leq C\gamma^2\max\{\|\nabla \eta\|^2_{L^{\infty}(S)},\|\eta\|^2_{L^{\infty}(S)}\} \max\left\{1, \|u\|^{\beta+2}_{L^{\gamma}(S)}\right\}.
 \label{eq_test4}
\end{equation}
Now, it is time to select the cut-off function $\eta$ in (\ref{eq_test4}). Assume that $0<r<R\leq 1$. 
Using the Aleksandrov maximum principle \cite[Theorem 1.4.2]{G01}, we find that 
\begin{equation}
\label{drR}
\dist(S_r,\p S_R)\geq c(\lambda,\Lambda)(R-r)^{\alpha},~\alpha=2.
\end{equation}
Indeed, by subtracting $R+ \varphi(x_0)+ \nabla\varphi(x_0)\cdot (x-x_0)$ from $\varphi(x)$, we can assume that $\varphi=0$ on $\p S_R$. Thus, $\varphi=-(R-r)$ on $\p S_r$. By the Aleksandrov maximum principle,
we have for any $x\in \p S_r$
\begin{eqnarray*}
(R-r)^2 =|\varphi(x)|^2 &\leq& C \dist(x,\p S_R) \diam(S_R)\int_{S_R} \det D^2 \varphi (x) dx\\
&\leq& C(\Lambda) \dist(x,\p S_R) \diam(S_R)|S_R| \leq C(n,\Lambda) \dist(x,\p S_R).
\end{eqnarray*}
Therefore, we obtain (\ref{drR}) as claimed.

With (\ref{drR}), 
we can choose a cut-off function $\eta\equiv 1$ in $S_r$, $\eta=0$ outside $S_R$, $0\leq \eta\leq 1$
and $$\|\nabla \eta\|_{L^{\infty}(S)}\leq \frac{C(\lambda,\Lambda)}{(R-r)^{\alpha}}.$$ 
It follows from (\ref{eq_test4}) that
\begin{eqnarray}
\label{ineq_iter}
\max\{1, \|u\|_{L^{2\gamma}(S_r)}\}&\leq& [C\gamma^2 (R-r)^{-2\alpha}]^{\frac{1}{\beta+ 2}}\max\{1, \|u\|_{L^{\gamma}(S_R)}\}
\nonumber\\ &=&[C\gamma^2 (R-r)^{-2\alpha}]^{\frac{\e+1}{\e \gamma}}\max\{1,\|u\|_{L^{\gamma}(S_R)}\}.
\end{eqnarray}
Now, for a  nonnegative integer $j$, set 
$$
r_{j}:= \frac{1}{2} + \frac{1}{2^j}, \quad \gamma_{j}:= 2^j\gamma_0, \quad \text{where} \quad \gamma_0:= p_0=2(\e+1)/\e.
$$
Then $r_j- r_{j+1}= \frac{1}{2^{j+1}}$. Applying the estimate (\ref{ineq_iter}) to $R= r_j$, $r= r_{j+1}$ and $\gamma=\gamma_j$, we get
\begin{eqnarray*}
\max\{1,\|u\|_{L^{\gamma_{j+1}}(S_{r_{j+1}})}\} &\leq& [C\gamma_j^2 (r_j-r_{j+1})^{-2\alpha}]^{\frac{\e+1}{\e \gamma_j}}\max\{1, \|u\|_{L^{\gamma_j}(S_{r_j})}\}
\\&\leq& [C\gamma_0^2 16^{j(\alpha+1)}]^{\frac{\e+1}{\e 2^j\gamma_0}}\max\{1, \|u\|_{L^{\gamma_j}(S_{r_j})}\}.
\end{eqnarray*}
By iterating, we obtain for all nonnegative integer $j$
\begin{eqnarray*}
\displaystyle\max\{1,\|u\|_{L^{\gamma_{j+1}}(S_{r_{j+1}})}\}
&\leq& \displaystyle C^{\sum_{j=0}^{\infty} \frac{\e+1}{\e 2^j\gamma_0}}16^{\sum_{j=0}^{\infty} \frac{(\e+1)(\alpha+1)j}{\e\gamma_0 2^j} }\max\{1,\|u\|_{L^{\gamma_0}(S_{r_0})}\}\\ &=:& C_0 
\max\{1, \|u\|_{L^{\gamma_0}(S_{r_0})}\} = C_0\max\{1,  \|u\|_{L^{p_0}(S)}\}=C_0.
\end{eqnarray*}
Letting $j\rightarrow \infty$ in the above inequality, we obtain (\ref{uboundC}).  The proof of Proposition \ref{int_thm_v2} is complete.
\end{proof}

\section{Regularity for Polar Factorization of time-dependent maps in two dimensions}
\label{app_sec}

In this section we use Theorem \ref{Holder_int_thm} to prove the local H\"older regularity for the polar 
factorization of time-dependent maps in two dimensions with densities bounded away from zero and infinity. 
Our applications improve previous work by Loeper who considered the cases of densities sufficiently close to a positive constant.  Our presentation in this section closely follows \cite{Loe}.

Throughout, we use $|E|$ to denote the Lebesgue measure of a Lebesgue measurable set $E\subset\R^n$.
\subsection{Polar factorization}
Let us start with the polar factorization. The polar factorization of vector-valued mappings was introduced by Brenier in his influential paper \cite{Br}. 
He showed that given a bounded open set $\Omega$ of $\R^n$ (which we can assume that $0\in\Omega$)
such that $|\partial\Omega|=0$, 
every Lebesgue measurable 
mapping $X\in L^2(\Omega; \R^n)$  satisfying the 
non-degeneracy condition 
\begin{equation}\label{2N}
|X^{-1}(B)|=0 ~\text{for all measurable~} B\subset\R^n \text{ with } |B|=0
\end{equation}
can be factorized into:
\begin{equation}\label{2polar}
X=\nabla P\circ  g,
\end{equation}
where $P$ is a convex function defined uniquely up to an additive constant and $g:\Omega\rightarrow\Omega$ is a Lebesgue-measure preserving mapping of $\Omega$; that is, 
\begin{equation}\label{2meas-pres}
\int_{\Omega}f(g(x)) \ dx =\int_{\Omega}f(x) \ dx~\text{for all  }f \in C_b(\Omega), 
\end{equation}
where $C_b$ is the set of bounded continuous functions.

If $\mathcal{L}_{\Omega}$ denotes the Lebesgue measure of $\Omega$, the push-forward of  $\mathcal{L}_{\Omega}$ by $X$, that we denote $X\# \mathcal{L}_{\Omega}$, is the measure $\rho$ defined by
\begin{equation}\label{2defrho}
 \int_{\R^n} f(x) d\rho(x)=\int_{\Omega} f(X(y))dy~\text{for all } f\in C_b(\R^n). 
\end{equation}
One can see that the condition (\ref{2N}) is equivalent to the absolute continuity of $\rho$ with respect to the Lebesgue measure, or $\rho\in L^1(\R^n, dx)$.

By (\ref{2polar})-(\ref{2defrho}), $P$ is a {\it Brenier solution} to the
Monge-Amp\`ere equation
$$ 
\rho(\nabla P(x))\det D^2 P(x)=1 ~\text{in }\Omega,
$$
that is,
\begin{equation}\label{2defPhi}
\int_{\Omega}\Psi(\nabla P(y))dy =\int_{\R^n} \Psi(x)d\rho(x) \text{ for all } \Psi\in C_b(\R^n).
\end{equation}
Moreover, $P$ satisfies the following second boundary condition
\begin{equation}
\label{SBC}
\nabla P(\Omega)=\Omega^{\ast}
\end{equation}
where $\Omega^{\ast}$ is the support of $\rho$.

Let us denote by 
$P^{\ast}$ the Legendre transform of $P$; that is, $P^{\ast}$ is defined by
\begin{equation}\label{2deflegendre}
P^{\ast}(y) = \sup_{x\in \Omega} \{x\cdot y -P(x)\}.
\end{equation}
Then $P^{\ast}$ is a Brenier solution to 
the Monge-Amp\`ere equation 
$$
\det D^2 P^{\ast}(x)=\rho(x) ~\text{in } \Omega^{\ast},
$$
that is,
\begin{equation}
\label{2defPsi}
 \int_{\R^n} f(\nabla P^{\ast}(x))d\rho(x)=\int_\Omega f(y)dy \text{ for all }f\in C_b(\Omega).
\end{equation}
Moreover, $P^{\ast}$ satisfies the following second boundary condition
\begin{equation}
\label{SBC2}
\nabla P^{\ast}(\Omega^\ast)=\Omega.
\end{equation}
Note that the Brenier solution to the Monge-Amp\`ere equation
is in general not the Aleksandrov solution. However, Caffarelli showed in \cite{C2} that if $\Omega^{\ast}$ is 
convex then $P$ is an Aleksandrov solution to $\rho(\nabla P(x))\det D^2 P(x)=1. $

In \cite{Loe}, Loeper investigated the regularity of the polar factorization of time-dependent maps $X_t\in L^2(\Omega;\R^n)$ where $t$ belongs to some open interval
$I\subset\R$. The open, bounded set $\Omega \subset \R^n$ is now assumed further to be smooth, strictly convex and has Lebesgue measure one.

As above, we assume that for each $t\in I$, $X_t$ satisfies (\ref{2N}).
For each $t\in I$, let $d\rho_t=X_t\# \mathcal{L}_{\Omega}$ be as in (\ref{2defrho}). Then, from $|\Omega|=1$, we find that $\rho_t$ is a probability measure on $\R^n$.

Let $P_t$ and $P_t^{\ast}$ be as in (\ref{2defPhi}) and (\ref{2defPsi}). 
Since $P_t$ is defined up to a constant, we impose the condition
\begin{equation}\label{mean:Phi}
\int_\Omega P_t( x) \dx = 0 \quad \text{for all }t \in I
\end{equation}
to guarantee uniqueness. 
Consider the function $g_t$ in the polar decomposition of $X_t$ as in (\ref{2polar}), that is
\begin{equation}\label{fact:X:Phi:g}
X_t(x)= \nabla P_t(g_t(x)) \quad \text{for all }x \in \Omega,
\end{equation}
$g_t: \Omega \to \Omega$ is a Lebesgue-measure preserving mapping. 
For each $t \in I$, the convex function $P_t$ is a Brenier solution to the following Monge-Amp\`ere equation in $\Omega$
\begin{equation}\label{MA:Phi}
\rho_t(\nabla P_t) \det D^2 P_t = 1.
\end{equation}
On the other hand, $P_t^\ast$ is a Brenier solution to 
the Monge-Amp\`ere equation 
\begin{equation}
\det D^2 P_t^{\ast} = \rho_t~\text{in}~\Omega_t^*:= \nabla P_t(\Omega)
\label{psit_eq}
\end{equation}
with the boundary condition $\nabla P_t^{\ast}(\Omega_t^\ast)=\Omega$.

In \cite{Loe}, Loeper investigated the regularity  of the curve $t\rightarrow (g_t, P_t, P_t^{\ast})$ under the assumptions:
\begin{myindentpar}{1cm}
$\bullet$ $X_t$ and $\partial_t X_t$ belong to $L^\infty(I \times \Omega)$;\\
$\bullet$ $\rho_t$ belongs to $L^{\infty}(I\times \R^n)$. 
\end{myindentpar}
We note that in this case $$\Omega_t^{\ast}\subset B_{R^{\ast}}(0)\equiv B_{R^*}~ \text{where } R^{\ast}=\|X_t\|_{L^{\infty}(I\times \Omega)}.$$
Several results were obtained in \cite{Loe}.
Among other results, Loeper proved (see \cite[Theorems 2.1, 2.3 and 2.3]{Loe}):
\begin{myindentpar}{1cm}
 1. For a.e. $t\in I$, $\p_t g_t$ and $\p_t \nabla P_t$ are bounded measures in $\Omega$. In particular, letting $\mathcal{M}(\Omega)$ denote 
 the set of vector-valued bounded measures on $\Omega$, we have
\begin{equation*}
\|\partial_t \nabla P_t\|_{\mathcal{M}(\Omega)}\leq C(R^{\ast},n,\Omega)\|\rho_t\|^{\frac{1}{2}}_{L^{\infty}(I\times B_{R^{\ast}})}
\|\partial_t X_t\|_{L^{\infty}(I\times \Omega)}.
\end{equation*}
2. The H\"older continuity of  $\partial_t P_t^{\ast}$ 
under the additional assumption that the density $\rho_t$ is sufficiently close to a positive constant.\\
3. The H\"older continuity of  $\p_t P_t$  
under the additional assumptions that $\Omega_t^{\ast}$ is convex  and the density $\rho_t$ is sufficiently close to a positive constant.
\end{myindentpar}

In Theorem \ref{H_thm_1} below, we are able to obtain 
the local H\"older regularity for $\partial_t P_t^{\ast}$ and $\p_t P_t$ {\bf in two dimensions} without assuming the closeness to 1 of the density 
$\rho_t$. Instead, we just assume it to be bounded away from zero and infinity, that is, for some positive constants $\lambda$ and $\Lambda$, we have
\begin{equation}\label{rho:lambdas}
\lambda \leq \rho_t\leq \Lambda~\text{on }\Omega_t^\ast~\text{for all }t\in I.
\end{equation}

\begin{thm}
[H\"older regularity of polar 
factorization of time-dependent maps in two dimensions]
\label{H_thm_1} 
Let $n=2$. Let $\Omega$ be a smooth, strictly convex set in $\R^2$ with $|\Omega|=1$. Let $I\subset\R$ be an open interval.
Assume that $X_t$ satisfies (\ref{2N}), $X_t$ and $\p_t X_t$ belong to $L^\infty(I \times \Omega)$ with
$R^{\ast}=\|X_t\|_{L^{\infty}(I\times \Omega)}$. Let $d\rho_t=X_t\# \mathcal{L}_{\Omega}$.
 Assume that the density $\rho_t$ satisfies \eqref{rho:lambdas}.
 Let $P_t$ and $P_t^{\ast}$ be as in (\ref{2defPhi}) and (\ref{2defPsi}). Assume (\ref{mean:Phi}) holds.
 Then 
\begin{myindentpar}{1cm}
(i) For any $\omega^\ast \subset\subset \Omega_t^{\ast}$, $\partial_t P_t^{\ast} \in C^{\alpha}(\omega^*)$ for some constant $\alpha =\alpha(\lambda,\Lambda)
\in (0,1)$  with 
$$\|\p_t P_t^{\ast}\|_{C^{\alpha}(\omega^\ast)}\leq C(\lambda,\Lambda, R^{\ast}, \emph{dist} (\omega^\ast, \p\Omega_t^\ast),  \|\nabla P_t^\ast\|_{C^{\delta}(\hat\omega)}, \Omega)$$
where $\delta =\delta(\lambda,\Lambda)
\in (0,1)$. 
Here $\hat w$ is the set of points in $\Omega_t^{\ast}$ of distance at least $\frac{1}{2}(\omega^\ast, \p\Omega_t^\ast)$ from $\p\Omega_t^\ast$.\\
(ii) If $\Omega_t^{\ast}$ is convex then for any $\omega\subset\subset\Omega$, $\p_t P_t \in C^{\beta}(\omega)$  for some constant
$\beta=\beta(\lambda,\Lambda) \in (0,1)$ with
$$\|\p_t P_t\|_{C^{\beta}(\omega)}\leq C(\lambda,\Lambda, R^{\ast}, \emph{dist} (\omega^\ast, \p\Omega_t^\ast), \emph{dist}(\omega,\p\Omega),\Omega_t^{\ast},\Omega).$$
\end{myindentpar}
\end{thm}

\begin{proof}[Proof of Theorem \ref{H_thm_1}] 
Since $P_t^\ast$ is a Brenier solution to (\ref{psit_eq}) on $\Omega_t^\ast$ with the boundary condition $\nabla P_t^\ast(\Omega_t^\ast)=\Omega$ where $\Omega$ is convex, we deduce from (\ref{rho:lambdas})
and Caffarelli's regularity results for the Monge-Amp\`ere equation \cite{C1,C2} that $P_t^\ast$ is locally $C^{1,\delta}$ with $\delta\in (0,1)$ depends only on $\Lambda/\lambda$. Note that $P_t$ is not $C^1$ in general.\\
(i) In \cite[Section 4]{Loe}, Loeper constructed an adequate smooth approximation of the polar factorization problem for time-dependent maps when 
 $X_t$ and $\partial_t X_t$ belong to $L^\infty(I \times \Omega)$ and  $\rho_t$ belongs to $L^{\infty}(I\times \R^n)$. Thus, we will assume in what follows, all functions $P_t$ and $P_t^\ast$ are 
 smooth. However, our estimates will 
not depend on the
smoothness.

Differentiating both sides of (\ref{psit_eq}) with 
respect to t, we obtain the following linearized Monge-Amp\`ere for $\p_t P_t^{\ast}$
$$
\nabla \cdot (M_{P_t^{\ast}} \nabla (\partial_t P_t^{\ast})) = \p_t \rho_t
$$
where $M_{P_t^{\ast}}$ represents the matrix of cofactors of $D^2 P_t^{\ast}$; that is, $M_{P_t^{\ast}}=(\det D^2 P_t^{\ast})(D^2 P_t^\ast)^{-1}$.\\
Loeper's important insight (see \cite[Section 4]{Loe}) is that $\rho_t$ satisfies a continuity equation of the form
$$
\partial_t\rho_t + \div(\rho_t v_t) = 0
$$
where $v_t$ is a smooth vector field on $\R^2$
and
\begin{equation}\label{bound:v}
\norm{v_t }_{L^\infty(\R^2, d\rho_t)} \leq \norm{\p_t X_t}_{L^\infty(\Omega)}.
\end{equation}
 In fact, $v_t$ can be computed explicitly via $g_t$, $P_t^{\ast}$ and $P_t$ by the formula (see, \cite[p. 345]{Loe} )
\begin{equation*}
v_t= \p_t \nabla P_t(\nabla P_t^{\ast}) + D^2 P_t w_t(\nabla P_t^{\ast})~\text{where}~ w_t(x)=\p_t g_t(g_t^{-1}(x)).
\end{equation*}
Therefore,  $\p_t P_t^{\ast}$ satisfies the linearized Monge-Amp\`ere equation
\begin{equation}
\label{LMApsi}
\nabla \cdot (M_{P_t^{\ast}} \nabla (\partial_t P_t^{\ast})) = \div (-\rho_t v_t).
\end{equation}
We now divide the proof into several steps.\\
{\it Step 1:} Given a section $S_{P_t^{\ast}}(x_0, 4h_0) \subset \subset \Omega_t^*$ and $x \in S_{P_t^\ast}(x_0, h_0)$, 
by Theorem \ref{Holder_int_thm} applied to (\ref{LMApsi}), we can find a constant $\gamma\in (0,1)$ depending only on $\lambda$ and $\Lambda$ such that
 \begin{multline}
 \label{HPt}
  |\p_t  P_t^{\ast}(x)-\p_t  P_t^{\ast}(x_0)|\\ \leq  C(\lambda,\Lambda, \diam(\Omega_t^{\ast}), h_0)  \left( \|\p_t P_t^{\ast}\|_{L^{2}(S_{P_t^{\ast}}(x_0, 2h_0))}+ 
 \|\rho_t v_t \|_{L^{\infty}(S_{P_t^{\ast}}(x_0, 2h_0))} \right)|x-x_0|^{\gamma}.
 \end{multline}
Since $P_t$ is the Legendre transform of $P_t^\ast$, we have
$$P_t(x)+ P_t^{\ast}(\nabla P_t(x))= x\cdot \nabla P_t(x)~\text{in }\Omega.$$
Differentiating both sides of the above equation with respect to $t$, and using $\nabla P_t^\ast(\nabla P_t(x))=x$, we obtain
\begin{equation}\label{ptPP}
\partial_t P_t (x)= - \partial_t P_t^{\ast}(\nabla P_t(x))~~\text{in }\Omega.
\end{equation}
By changing variables $y:=\nabla P_t(x)$ we have from (\ref{MA:Phi}) that
\begin{equation}\label{L2:psi:phi}
\int_{\Omega_t^*} |\partial_t P_t^{\ast}(y)|^2 d\rho_t(y) = \int_{\Omega} |\partial_t P_t(x)|^2 \dx.
\end{equation}
From the condition \eqref{mean:Phi} and the Poincar\'e-Sobolev embedding theorem $W^{1,1}(\Omega) \hookrightarrow L^{2}(\Omega)$ for the convex set $\Omega\subset \R^2$,  we can find a positive 
constant $C(\Omega) > 0$ depending only on $\Omega$ such that
\begin{equation}\label{CP:omega}
\|\p_t P_t\|_{L^2(\Omega)} \leq C(\Omega) \int_\Omega |\nabla \partial_t P_t(x)| \dx.
\end{equation}
By \cite[Theorem 2.1]{Loe}, we have
$$\int_\Omega |\nabla \partial_t P_t(x)| \dx \leq C(R^\ast,\Omega)\|\rho_t\|^{\frac{1}{2}}_{L^{\infty}(I\times B_{R^{\ast}})}
\|\partial_t X_t\|_{L^{\infty}(I\times \Omega)}.$$
Therefore, by combining \eqref{L2:psi:phi} and \eqref{CP:omega}, we obtain
\begin{eqnarray}
\label{L2Pt}
\|\partial_t P_t^{\ast}\|_{L^{2}(S_{P_t^{\ast}}(y_0, 2h_0), d\rho_t)} \leq \|\p_t P_t\|_{L^2(\Omega)}
\leq C(R^\ast,\Omega)\|\rho_t\|^{\frac{1}{2}}_{L^{\infty}(I\times B_{R^{\ast}})}
\|\partial_t X_t\|_{L^{\infty}(I\times \Omega)}.
\end{eqnarray} 
Putting (\ref{bound:v}), (\ref{L2Pt}) and (\ref{HPt}) all together, we have
\begin{equation*}
 |\p_t  P_t^{\ast}(x)-\p_t  P_t^{\ast}(x_0)|\leq  C(\lambda,\Lambda, R^{\ast}, h_0,\Omega) \|\partial_t X_t\|_{L^{\infty}(I\times \Omega)}|x-x_0|^{\gamma}.
\end{equation*}
{\it Step 2:} Suppose now $\omega^\ast\subset\subset\Omega^{\ast}_t$. Let $\hat w$ be
the set of points in $\Omega_t^{\ast}$ of distance at least $\frac{1}{2}(\omega^\ast, \p\Omega_t^\ast)$ from $\p\Omega_t^\ast$.
Then, there is $h_0>0$ depending only on $\lambda,\Lambda, \diam(\Omega)$, 
$ \|\nabla P_t^\ast\|_{C^{\delta}(\hat \omega)}$ and $\dist(\omega^\ast,\p\Omega_t^\ast)$
such that
$$S_{P_t^\ast}(x_0, 4h_0)\subset\subset\hat \omega.$$
Indeed, the strict convexity of $P_t^\ast$ implies the existence of $h_0>0$ such that $S_{P_t^\ast}(x_0, 4h_0)\subset\subset\hat\omega.$ For any of these sections, we first 
use the $C^{1,\delta}$ property of $P_t^\ast$ to deduce that $$S_{P_t^\ast}(x_0, 4h_0)\supset B_{c_1 h_0^{\frac{1}{1+\delta}}}(x_0), \text{where }c_1=  
\|\nabla P_t^\ast\|^{-\frac{1}{1+\delta}}_{C^{\delta}(\hat\omega)}.$$
The volume estimates (\ref{vol-sec1}) give
$|S_{P_t^\ast}(x_0, 4h_0)|\leq C(\lambda,\Lambda)h_0.$
Using the convexity of $S_{P_t^\ast}(x_0, 4h_0)$, we easily find that
$$S_{P_t^\ast}(x_0, 4h_0) \subset B_{C(\lambda, \Lambda, c_1) h_0^{\frac{\delta}{1+\delta}}}(x_0).$$
Thus,  $h_0>0$ can be chosen to depend only on $\lambda,\Lambda, \diam(\Omega)$, $ \|\nabla P_t^\ast\|_{C^{\delta}(\hat \omega)}$ 
and $\dist(\omega^\ast,\p\Omega_t^\ast)$ so that $S_{P_t^\ast}(x_0, 4h_0)\subset\subset\hat \omega\subset\subset\Omega_t^\ast$.

Now, we can conclude from {\it Step 1} and {\it Step 2} the H\"older continuity of $\p_t P_t^{\ast}$ as asserted in (i).\\
(ii) Because
 $P_t$ is the Brenier solution to the Monge-Amp\`ere equation \eqref{MA:Phi} with the second boundary condition $\nabla P_t (\Omega)\subset \Omega_t^{\ast}$, and 
 $\Omega_t^{\ast}$ is convex, it is also
 the Aleksandrov solution as proved by Caffarelli \cite{C2}.
The hypothesis \eqref{rho:lambdas} yields that, in the sense of Aleksandrov, $$\Lambda^{-1} \leq \det D^2 P_t \leq \lambda^{-1}~ \text{in }\Omega.$$ Hence, Caffarelli's global
regularity result \cite{C3}
yields $P_t\in C^{1,\delta}(\overline{\Omega})$ and $P^\ast_t\in C^{1,\delta}(\overline{\Omega_t^\ast})$. 
These combined with \eqref{ptPP} and (i) give the conclusion of (ii) with $\beta=\gamma\delta$. 
The proof of Theorem \ref{H_thm_1} is complete.
\end{proof}
\subsection{Polar factorization of time-dependent maps on the torus} 
The polar factorization of maps on general Riemannian manifolds has been treated by McCann \cite{Mc}, and also in the  particular case of the flat torus $\T^n= \R^n/\Z^n$ by 
Cordero-Erausquin \cite{CE}. Given a mapping $X$ from $\T^n$ into itself, then under the non-degeneracy condition (\ref{2N}), there is a unique pair $(P, g)$ such that
\begin{myindentpar}{1cm}
 1. $X= \nabla P \circ g,$\\
 2. $g: \T^n \to \T^n$ is a Lebesgue-measure preserving map,\\
 3. $P: \R^n \to \R$ is convex and  $P - \frac{|x|^2}{2}$ is $\Z^n$-periodic.
 \end{myindentpar}
The analogue of Theorem \ref{H_thm_1} for the regularity of polar factorization of time-dependent maps on the two-dimensional torus is the following theorem.
\begin{thm} 
[H\"older regularity of polar 
factorization of time-dependent maps on the 2D torus]
\label{peri_thm} Let $I\subset\R$ be an open interval. Suppose that $X_t,\p_t X_t \in L^\infty(I \times \T^2)$
where $X_t:\T^2\rightarrow \T^2$ satisfies (\ref{2N}) for all $t\in I$.
Let $d\rho_t=X_t\#  \mathcal{L}_{\T^2}$. Suppose that
$ \lambda \leq \rho_t\leq \Lambda$ on $\T^2$ for some positive constants $\lambda, \Lambda$ and  all $t \in I $. 
Let $P_t$ and $P_t^{\ast}$ be as in (\ref{2defPhi}) and (\ref{2defPsi}) where $\Omega$ is now replaced by $\T^2$. 
Then, there 
exist $\alpha=\alpha(\lambda, \Lambda) \in (0,1)$ and $C=C(\lambda,\Lambda)>0$ such that  
$$\|\p_t P_t^{\ast}\|_{L^{\infty}(I, C^{\alpha}(\T^2))} + \|\p_t P_t\|_{L^{\infty}(I, C^{\alpha}(\T^2))} \leq C.$$
\end{thm}
\begin{proof}[Proof of Theorem \ref{peri_thm}] 
As in the proof of Theorem \ref{H_thm_1}, the density $\rho_t$ satisfies the continuity equation
$
\partial_t\rho_t + \div(\rho_t v_t) = 0
$
where $v_t$ is a bounded vector field with
$\norm{v_t }_{L^\infty(\T^2)} \leq \norm{\p_t X_t}_{L^\infty(\T^2)}.
$
This vector field $v_t$ is similar to the vector field $U_t$ in Theorem \ref{GSthm} where the only information we used is its uniform boundedness in $t$.
Moreover, for all $t>0$, we still have (see \cite[Theorem 4.5]{F})
$$\int_{\T^2} \rho_t(x) |\p_t \nabla P_t^{\ast}(x)|^{1+\kappa}dx\leq C(\lambda,\Lambda)~\text{where } \kappa= \kappa (\lambda,\Lambda)>0.$$
Thus, Theorem \ref{peri_thm} follows from the same arguments as in the proof of Theorem \ref{GSthm}.
\end{proof}

\section{Green's function and The Monge-Amp\`ere Sobolev inequality}
\label{auxi_sec}
In this section, we prove  Propositions \ref{G2_thm}, \ref{sob_ineq} and \ref{G1_thm}. 
Recall that in these propositions and this section, $\Omega \subset \R^2$ is a bounded convex set with nonempty interior and $\varphi \in C^2(\Omega)$ is a convex function such that
$$
\lambda\leq \det D^2\varphi \leq \Lambda
$$
for some positive constants $\lambda$ and $\Lambda$. 

Given a section $S=S_\varphi(x_0, h)\subset\subset\Omega$ of $\varphi$, we let 
$g_S(x, y)$ be Green's function of the operator $L_\varphi:=-\p_j (\Phi^{ij} \p_i)=-\Phi^{ij}\p_{ij}$ on $S$, as in (\ref{Green:S}).

To prove  Proposition \ref{G1_thm}, 
we recall the following fact regarding Green's function in 2D.
\begin{lem}
\label{smallV}
Suppose that $S= S_\varphi(x_0, h)\subset\subset\Omega$ and $S_{\varphi}(y, \eta h)\subset S_\varphi(x_0, h)$ for some $\eta\in (0, 1)$. Then
\begin{myindentpar}{1cm}
(i)  (\cite[Section 6]{L2}) For all  $x\in \p  S_{\varphi}(y,\frac{\eta}{2}h)$, we have
$g_S(x, y)\leq C(\lambda,\Lambda,\eta).$\\
(ii) (\cite[Lemma 3.2]{L}) For all $0<h_1<\frac{\eta}{4}h$, we have
$$\max_{x\in\p S_\varphi(y, h_1)} g_S(x, y)\leq C (\lambda,\Lambda) + \max_{z\in \p S_\varphi(y, 2h_1)} g_S(z, y).
$$
\end{myindentpar}
\end{lem}
For reader's convenience, we explain how to derive Lemma \ref{smallV}(i) from \cite{L2}. By Lemma 6.1 in \cite{L2} and the volume estimates (\ref{vol-sec1}), we obtain
$$\int_{S} g_S (x, y) dx\leq C(\lambda)|S|\leq C(\lambda,\Lambda)h.$$
Applying Lemma 6.2 in \cite{L2} to $g_S(\cdot, y)$ in $S_{\varphi}(y,\eta h)$, we obtain for all $x\in \p  S_{\varphi}(y,\frac{\eta}{2}h)$ the estimate
$$g_S(x, y)\leq C(\lambda,\lambda) h (\eta h)^{-1} \leq C(\lambda,\Lambda,\eta).$$

Lemma \ref{smallV} implies the following lemma:
\begin{lem} 
\label{g_decay_lem}
Assume that $S= S_\varphi(x_0, h)\subset\subset\Omega$ and $S_{\varphi}(y, \eta h)\subset S_\varphi(x_0, h)$ for some $\eta\in (0, 1)$. Let $\tau_0=C(\lambda,\Lambda, \eta) + C(\lambda,\Lambda)$ where $C(\lambda,\Lambda, \eta)$  and $ C(\lambda,\Lambda) $ are as in Lemma \ref{smallV}.
Then for all $\tau> \tau_0$, we have
 $$\{x\in S: g_{S}(x, y)>\tau\}\subset S_{\varphi}(y, \eta h 2^{-\tau/\tau_0}).$$
\end{lem}
\begin{proof}[Proof of Lemma \ref{g_decay_lem}]
If $\tau>\tau_0$ then by the maximum principle, we can find $h_1<\eta h/2$ such that
\begin{equation}
\label{ytau}
\{x\in S: g_{S}(x, y)>\tau\}\subset S_{\varphi}(y, h_1).
\end{equation}
Let $m$ be a positive integer such that $\eta h/2\leq 2^m h_1<\eta h$.
Iterating Lemma \ref{smallV}(ii), we find that
$$\max_{x\in\p S_\varphi(y, h_1)} g_S(x, y)\leq m C(\lambda,\Lambda) + \max_{x\in\p S_\varphi(y, 2^m h_1)} g_S(x, y).$$
The maximum principle and Lemma \ref{smallV}(i) give
$$\max_{x\in\p S_\varphi(y, 2^m h_1)} g_S(x, y) \leq \max_{x\in\p S_\varphi(y, \eta h/2)} g_S(x, y)\leq C(\lambda,\Lambda,\eta).$$
Hence, 
$$\max_{x\in\p S_\varphi(y, h_1)} g_S(x, y)\leq m (C(\lambda,\Lambda,\eta)+ C(\lambda,\Lambda))\leq m\tau_0\leq \tau_0 \log_2 (\eta h/h_1).$$
Thus, we obtain from (\ref{ytau}) the estimate
$\tau\leq  \tau_0 \log_2 (\eta h/h_1).$
It follows that 
$h_1\leq \eta h 2^{-\tau/\tau_0}.$
Lemma \ref{g_decay_lem} now follows from (\ref{ytau}).
\end{proof}
Now we are ready to give the proof of Proposition \ref{G1_thm}.
\begin{proof} [Proof of Proposition \ref{G1_thm}] Recall that $S=S_{\varphi}(x_0, h)$ with $S_{\varphi}(x_0, 2h)\subset\subset\Omega$.\\
{\it Step 1: Special case.} We first prove 
\begin{equation}
 \int_S g_S^p(x, y) dx \leq C(\lambda,\Lambda, p) h.
 \label{center_reduce}
 \end{equation}
in the special case 
that 
$S_{\varphi}(y, \eta h)\subset S_{\varphi}(x_0, h)
$
for some universal constant $\eta\in (0,1)$ depending only on $\lambda$ and $\Lambda$. In this case, by Lemma \ref{g_decay_lem}, we find that for all $\tau>\tau_0(\lambda,\Lambda)$,
$$\{x\in S: ~ g_S(x,y)> \tau \}\subset S_{\varphi}(y, \eta h 2^{-\tau/\tau_0}).$$
Using the upper bound on the volume of sections in (\ref{vol-sec1}), we find that 
$$|\{x\in S: ~ g_S(x,y)> \tau \}| \leq C(\lambda,\Lambda) h 2^{-\tau/\tau_0}.$$
It follows from the layer cake representation and the volume estimate $|S|\leq C(\lambda,\Lambda)h$ that
\begin{eqnarray*}
\int_S g_S^p(x, y) dx&=&\int_0^{\infty} p\tau^{p-1} |\{x\in S: ~ g_S(x,y)> \tau \}| d\tau\\ &\leq& \int_0^{\tau_0} p\tau^{p-1} |\{x\in S: ~ g_S(x,y)> \tau \}| d\tau +
\int_{\tau_0}^{\infty} p\tau^{p-1} |\{x\in S: ~ g_S(x,y)> \tau \}| d\tau
\\ &\leq &\tau_0^p |S| + \int_{\tau_0}^{\infty} p\tau^{p-1} C(\lambda,\Lambda)h 2^{-\tau/\tau_0} d\tau\leq C(p,\lambda,\Lambda)h.
\end{eqnarray*}
Hence (\ref{center_reduce}) is proved. 

{\it Step 2: General case.} It remains to prove the proposition for the general case $y\in S_{\varphi}(x_0, h)$. By \cite[Theorem 3.3.10(i)]{G01}, there is a universal constant $\eta\in (0,1)$ depending 
only on $\lambda$ and $\Lambda$ such that
for all $y\in S_{\varphi}(x_0, h)$, we have 
\begin{equation}S_{\varphi}(y, \eta h)\subset S_{\varphi}(x_0, \frac{3}{2} h). 
 \label{yaway2}
\end{equation}
Thus, for any $x\in S_{\varphi}(x_0, h)$, we have
$g_{S_{\varphi}(x_0,  \frac{3}{2} h)}(x, y)\geq g_{S_{\varphi}(x_0, h)}(x, y)$ by the maximum principle.
It follows that 
\begin{equation*}
\int_{S_{\varphi}(x_0, h)} g_{S_{\varphi}(x_0, h)}^p(x, y) dx \leq \int_{S_{\varphi}(x_0,  \frac{3}{2}h)} g_{S_{\varphi}(x_0, \frac{3}{2} h)}^p(x, y) dx
\leq C(p, \lambda,\Lambda)h.
\end{equation*}
In the last inequality, we applied the estimate (\ref{center_reduce}) to the section $S_{\varphi}(x_0, \frac{3}{2} h)$ and the point $y$ in $S_{\varphi}(x_0, \frac{3}{2} h)$ 
that satisfies (\ref{yaway2}). The proof of Proposition \ref{G1_thm} is now complete.
\end{proof}
Note that, Proposition \ref{G2_thm} is closely related to \cite[Theorem 1.1(iii)]{L}. 
For reader's convenience
we provide the detailed proof of Proposition \ref{G2_thm}
 here.
 \begin{proof}[Proof of Proposition \ref{G2_thm}] Recall that $S=S_{\varphi}(x_0, h)$.
   Fix $y\in S$ and set
  $$v(x) := g_S(x, y) + 1\quad \text{for all }x \in \overline{S}.
  $$
  Then $v\geq 1$ in $S$, $v=1$ on $\p S$ and $\p_i(\Phi^{ij}v_j)=-\delta_y$ in $S$.
  Let $\hat S= S_{\varphi}(x_0, 2h)$ and extend $v$ to be $1$ in $\hat S\backslash S$.
  Then $v\geq 1$ in $\hat S$. Using the divergence-free property of $(\Phi^{ij})$, we have $\p_i(\Phi^{ij}v_j)=-\delta_y$ in $\hat S$. \\
 {\it Step 1: Integral bound for $\log v$.} We show that
   \begin{equation}\int_S \Phi^{ij} (\log v)_i (\log v)_j dx\equiv \int_S \Phi^{ij} v_i  v_j\frac{1}{ v^2} dx \leq C(\lambda,\Lambda).
   \label{log_ineq}
  \end{equation}
 Indeed,  given $w \in C^1_0 (\hat S)$, 
 multiply the inequality
  $\p_i (\Phi^{ij} v_j)\leq 0$ by $\frac{w^2}{v}\geq 0$ and then integrate by parts to get
  $$0\geq - \int_{\hat S} \Phi^{ij} v_j \p_i \left(\frac{w^2}{ v}\right)dx =- \int_{\hat S} 2\Phi^{ij}  w_i v_j \frac{w}{ v} dx +  \int_{\hat S} \Phi^{ij} v_i  v_j\frac{w^2}{ v^2} dx.$$
  By the Cauchy-Schwarz inequality, we obtain
  \begin{equation*}
   \int_{\hat S} \Phi^{ij} \frac{v_i v_j w^2}{ v^2} dx \leq \int_{\hat S} 2\Phi^{ij} \frac{w_i v_j w}{ v} dx \leq 2 \left(\int_{\hat S} \Phi^{ij}\frac{v_i v_j w^2}{ v^2}dx\right)^{1/2}
    \left(\int_{\hat S} \Phi^{ij} w_i  w_j dx\right)^{1/2}
  \end{equation*}
and therefore
\begin{equation}
  \int_{\hat S} \Phi^{ij} v_i  v_j\frac{w^2}{ v^2} dx\leq 4 \int_{\hat S} \Phi^{ij} w_i  w_j dx.
  \label{phivh}
\end{equation}
By choosing a suitable $0\leq w\leq 1$ as in the proof of Theorem 6.2 in Maldonado \cite{Mal4}, and using the volume estimates in (\ref{vol-sec1}), we 
obtain (\ref{log_ineq}). For completeness, we include a construction of $w$.

By subtracting $\varphi(x_0)+\nabla \varphi(x_0)\cdot (x-x_0)$ from $\varphi$, we can assume that $\varphi(x_0)=0$ and $\nabla \varphi(x_0)=0$. Therefore 
$\varphi\geq 0$ on $\hat S=S_{\varphi}(x_0, 2h)$. Let $\gamma:\R\rightarrow [0, 1]$ be a smooth function supported in $[0, 2]$ with $\gamma\equiv 1$ on $[0, 1]$ and $\|\gamma'\|_{L^{\infty}(\R)}\leq 10$. Let $w(x):=\gamma(\varphi(x)/h)$. Then 
$w\in C^1_0(\hat S)$ with $w\equiv 1$ on $S_{\varphi}(x_0, h)=S$
and
$$\nabla w(x) = \frac{1}{h}\gamma^{'}(\varphi(x)/h) \nabla \varphi(x).$$
Therefore, 
\begin{eqnarray*}
\int_{\hat S} \Phi^{ij} w_i w_j dx&=&\frac{1}{h^2}\int_{\hat S} (\gamma^{'}(\varphi(x)/h))^2 \Phi^{ij} \varphi_i\varphi_j dx
\leq \frac{\|\gamma'\|^2_{\infty}}{h^2} \int_{\hat S}  \Phi^{ij} \varphi_i\varphi_j dx \\
&\leq& \frac{100}{h^2}\int_{\hat S} \langle \Phi \nabla (2h-\varphi(x)), \nabla (2h-\varphi (x))\rangle dx.
\end{eqnarray*}
Integrating by parts the last term and using $\sum_{i=1}^2\p_i \Phi^{ij}=0$ for all $j=1,2$, we obtain
\begin{eqnarray*}
\int_{\hat S} \Phi^{ij} w_i w_j dx
&\leq& - \frac{100}{h^2}\int_{\hat S} \div[\Phi \nabla (2h-\varphi(x))]  (2h-\varphi (x)) dx
= \frac{100}{h^2}\int_{\hat S} \trace[\Phi D^2 \varphi]  (2h-\varphi (x)) dx\\
&=&  \frac{100}{h^2}\int_{\hat S} 2(\det D^2\varphi) (2h-\varphi (x)) dx \leq \frac{400\Lambda}{h}|\hat S|\leq C(\lambda,\Lambda).
\end{eqnarray*}
In the last inequality, we used the upper bound on volume of sections in (\ref{vol-sec1}) to get $|\hat S|\leq C(\lambda,\Lambda)h$.
Therefore, (\ref{log_ineq}) now follows from (\ref{phivh}) and the above inequalities.

{\it Step 2: $L^{1+\kappa}$ estimate for $v$.} By Proposition \ref{G1_thm}  and the inequality
$v^q(x) \leq C(q) ( g^q_S(x, y)+ 1),$
together with the volume bound on $S$, we find that 
$v\in L^q(S)$ for all $q<\infty$ with the bound
 \begin{equation}\|v\|_{L^q(S)}\leq C(\lambda,\Lambda, q) h^{\frac{1}{q}}.
  \label{Glq}
 \end{equation}
Next, we use the following inequality
$\Phi^{ij} v_{i}(x) v_{j}(x)  \geq \frac{\det D^2 \varphi |\nabla v|^2}{\Delta \varphi}$
whose simple proof can be found in \cite[Lemma 2.1]{CG97}. 
It follows from (\ref{log_ineq}) that
$$\int_S \frac{|\nabla v|^2}{\Delta\varphi}\frac{1}{v^2}dx\leq C(\lambda,\Lambda).$$
Now, for all $1<p<2$, using the Holder inequality to $|\nabla v|^p= \frac{|\nabla v|^p}{(\Delta \varphi)^{\frac{p}{2}}} \frac{1}{v^p} \left((\Delta\varphi)^{\frac{p}{2}} v^p\right)$
with exponents $\frac{2}{p}$ and $\frac{2}{2-p}$, we have
\begin{eqnarray*}
\|\nabla v\|_{L^p(S)} \leq \left[\int_S \frac{|\nabla v|^2}{\Delta\varphi}\frac{1}{v^2}dx\right]^{\frac{1}{2}} \left(\int_S(\Delta \varphi)^{\frac{p}{2-p}} 
v^{\frac{2p}{2-p}}dx\right)^{\frac{2-p}{2p}} \leq C(\lambda,\Lambda) \|(\Delta\varphi)v^2\|_{L^{\frac{p}{2-p}}(S)}^{\frac{1}{2}}.
\end{eqnarray*}
Let $\e_\ast = \e_\ast(\lambda,\Lambda)>0$ be as in (\ref{w21est}).
Let us fix any $0<\e<\e_\ast$ and 
$$p = \frac{2(1+\e)}{2+\e},~\text{that is},~ \frac{p}{2-p}= 1+\e.$$
Thus, recalling $h_0\geq h$, Lemma \ref{DPFS_lem} and (\ref{Glq}), we obtain
\begin{eqnarray*}\|\nabla v\|_{L^p(S)} \leq \|(\Delta\varphi)v^2\|^{\frac{1}{2}}_{L^{1+\e}(S)}
&\leq& \|\Delta\varphi\|^{\frac{1}{2}}_{L^{1+\e_\ast}(S)} \|v\|_{L^{\frac{2(1+\e_\ast)(1+\e)}{\e_\ast-\e}}(S)}
\\&\leq& C(\lambda,\Lambda,\e,\e_\ast) \|\Delta\varphi\|^{\frac{1}{2}}_{L^{1+\e_\ast}(S_{\varphi}(x_0, h_0))}  h^{\frac{\e_\ast-\e}{2(1+\e_\ast)(1+\e)}}
\\&\leq& C(\lambda,\Lambda,\diam(\Omega), h_0)h^{\frac{\e_\ast-\e}{2(1+\e_\ast)(1+\e)}} .
\end{eqnarray*}
The proof of Proposition \ref{G2_thm} is complete by choosing
$\kappa = p-1= \frac{\e}{2+\e}$ and $\kappa_1= \frac{\e_\ast-\e}{2(1+\e_\ast)(1+\e)}.$
\end{proof}

\begin{proof}[Proof of Proposition \ref{sob_ineq}]
Suppose that $S_\varphi(x_0, 2)\subset\subset\Omega$ and $S_\varphi(x_0, 1)$ is normalized.
Set $S:=S_\varphi(x_0, 1)$.
Let $g_S(x,y)$ be Green's function of $S$ with respect to $L_\varphi:=-\p_j(\Phi^{ij}\p_i)=-\Phi^{ij}\p_{ij}$ with pole $y\in S$, that is, $g_S(\cdot, y)$ is a positive solution of \eqref{Green:S}. 
By Proposition \ref{G1_thm}, 
for any $q>1$,  there exists a constant $K>0$, depending on $q$, $\lambda$ and $\Lambda$, such that for every $y\in S$ we have
\begin{equation}\label{dist-est-Green2}
|\{x\in S: ~ g_S(x,y)> \tau \}|\leq K \tau^{-\frac{q}{2}}~ \text{ for all }\tau>0.
\end{equation}

As the operator $L_\varphi$ can be written in the divergence form with symmetric coefficients, we  infer from Gr\"uter-Widman \cite[Theorem~1.3]{GW}
that $g_S(x,y)=g_S(y,x)$ for all $x,y\in S$. This together with (\ref{dist-est-Green2})
 allows us to deduce that,
for every $x\in S$, there holds
\begin{equation*}
|\{y\in S: ~ g_S(x,y)> \tau \}|
=|\{y\in S: ~ g_S(y,x)>\tau \}| \leq K \tau^{-\frac{q}{2}}~  \text{for all } \tau >0.
\end{equation*}
Then, one can use Lemma 2.1 in Tian-Wang \cite{TiW08} to conclude \eqref{Sob:Phi}.
\end{proof}

\section{Proofs of Lemmas \ref{DPFS_lem} and \ref{res_eq_lem}}
\label{res_proof}

\begin{proof}[Proof of Lemma \ref{res_eq_lem}]
For $x\in \tilde S:=T^{-1}(S_\varphi(x_0, h))$, we have
$$D \tilde \varphi(x) = (\det A_h)^{-2/n} A_h^{t} D\varphi (Tx), ~D^2 \tilde \varphi(x) = (\det A_h)^{-2/n} A_h^{t}D^2\varphi(Tx) A_h.$$
and 
$$ D \tilde u (x)= A_h^{t} Du(Tx),~D^2 \tilde u(x) = A_h^{t}D^2 u (Tx)A_h.$$
In the variables $y:=Tx$ and $x\in \tilde S$, we have the relation $\nabla_x= A_h^{t}\nabla_y.$ Thus, letting $\langle, \rangle$ denote the inner product on $\R^n$, we have
\begin{eqnarray}
 \label{Fh}
 \nabla_x\cdot \tilde F(x)=\langle \nabla_x, \tilde F(x) \rangle= \langle A_h^{t}\nabla_y, (\det A_h)^{\frac{2}{n}} A_h^{-1} F(Tx) \rangle&=& \langle \nabla_y, (\det A_h)^{\frac{2}{n}} F(Tx) \rangle
 \nonumber\\&=&(\det A_h)^{\frac{2}{n}} (\nabla\cdot F)(Tx).
\end{eqnarray}
The cofactor matrix $\tilde \Phi= (\tilde \Phi^{ij}(x))_{1\leq i, j\leq n}$ of $D^2 \tilde \varphi(x)$ is related to $\Phi(Tx)$ and $A_h$ by
\begin{eqnarray}\tilde \Phi(x) = (\det D^2 \tilde \varphi(x))(D^2 \tilde \varphi(x))^{-1}&=&(\det D^2 \varphi(Tx)) (\det A_h)^{2/n} A_h^{-1} (D^2 \varphi(Tx))^{-1} (A_h^{-1})^{t}\nonumber\\ &=&
(\det A_h)^{2/n} A_h^{-1} \Phi(Tx) (A_h^{-1})^{t}. 
\label{tildeU}
\end{eqnarray}
Therefore, $$\tilde \Phi^{ij} \tilde u_{ij} (x)= \trace (\tilde \Phi (x)D^2\tilde u(x))= (\det A_h)^{2/n} \trace (\Phi(Tx) D^2 u(Tx))= (\det A_h)^{2/n}  \Phi^{ij} u_{ij}(Tx)$$
and hence, recalling (\ref{main_eq}) and (\ref{Fh}),
\begin{eqnarray*}
 \tilde \Phi^{ij} \tilde u_{ij} (x) = (\det A_h)^{2/n} (\nabla\cdot F)(Tx)= \nabla\cdot \tilde F=\div \tilde F(x)~\text{in}~\tilde S.
\end{eqnarray*}
Thus, we get (\ref{eqSh1}) as asserted. 

Next, we claim that
\begin{equation}
\dist(S_\varphi(x_0, h), \p S_\varphi(x_0, 2h)) \geq \frac{c(\lambda,\Lambda, n) h^{n/2}}{[\diam(\Omega)]^{n-1}}.
\label{dist_est_h}
\end{equation}
Indeed, let $\hat{\varphi}= \varphi-h$. Then, by our assumption that $\varphi|_{\p S_\varphi(x_0, h)}=0$, we have $\hat{\varphi}=0$ on $\p S_\varphi(x_0, 2h)$. Applying the Aleksandrov maximum principle (see \cite[Theorem 1.4.2]{G01}) to
$\hat{\varphi}$ on $S_\varphi(x_0, 2h)$, we have for any $x\in S_\varphi(x_0, h)$,
\begin{eqnarray*}h^n\leq |\hat{\varphi}(x)|^n&\leq& C(n) \dist(x,\p S_\varphi(x_0, 2h)) [\diam (S_\varphi(x_0, 2h))]^{n-1} \int_{S_\varphi(x_0, 2h)}\det D^2 \hat{\varphi} ~dx
\\&\leq& C(\Lambda, n)  \dist(x,\p S_\varphi(x_0, 2h)) [\diam (\Omega)]^{n-1} |S_\varphi(x_0, 2h)|
\\&\leq & C(\lambda, \Lambda, n)  \dist(x,\p S_\varphi(x_0, 2h)) [\diam (\Omega)]^{n-1} h^{n/2}
\end{eqnarray*}
where in the last inequality we used the volume estimates in (\ref{vol-sec1}). Thus, we obtain (\ref{dist_est_h}) as claimed.

Using (\ref{dist_est_h}) and the convexity of $\varphi$, we find that
\begin{equation}
\label{Dvar_est}
\|D\varphi\|_{L^{\infty}(S_\varphi(x_0, h))}\leq \frac{h}{\dist(S_\varphi(x_0, h), \p S_\varphi(x_0, 2h))} \leq C(\lambda,\Lambda, n,\diam(\Omega)) h^{1-\frac{n}{2}}.
\end{equation}
It follows that 
 $S_\varphi(x_0, h)\supset B(x_0, c_1h^{\frac{n}{2}})$ for some constant $c_1= c_1(n,\lambda,\Lambda,\diam(\Omega))$, which combined with (\ref{normSh}) implies that 
 \begin{equation} \|A_h^{-1}\|\leq C(n,\lambda,\Lambda,\diam(\Omega))h^{-\frac{n}{2}}.
  \label{Anorm}
 \end{equation}
On the other hand, by means of the volume estimates in (\ref{vol-sec1}), we find from (\ref{normSh}) that 
\begin{equation}
C(n,\lambda,\Lambda)^{-1} h^{n/2}\leq \det A_h
\leq C(n,\lambda,\Lambda) h^{n/2}.
\label{detAh}
\end{equation}
Hence (\ref{tFbound}) follows from (\ref{def:tildes}), (\ref{Anorm}) and (\ref{detAh}). 

Finally, using
$$\|\tilde u\|_{L^{q}(\tilde S)}=(\det A_h)^{-1/q} \|u\|_{L^q(S_{\varphi}(x_0, h))}$$
together with (\ref{detAh}), we obtain (\ref{tilde_u_Lq}).
\end{proof}
\begin{proof}[Proof of Lemma \ref{DPFS_lem}]
(i) Rescaling as in \eqref{def:tildes}, we have for all $x\in S_{\varphi}(x_0, h)$
$$D^2\varphi(x) = (\det A_h)^{\frac{2}{n}} (A_h^{-1})^t D^2 \tilde\varphi (T^{-1}x) A_h^{-1}.$$
Using the inequality
$\text{trace} (AB)\leq \|A\| \text{trace} (B)$
for nonnegative definite matrices $A, B$, we
thus have
$$\Delta\varphi (x)\leq \|A_h^{-1}\|^2  (\det A_h)^{\frac{2}{n}}\Delta\tilde\varphi(T^{-1}x).$$
By the $W^{2, 1+\e}$ estimate (\ref{w21est}) applied to $\tilde\varphi$ and its normalized section $\tilde S=T^{-1}(S_{\varphi}(x_0, h))$, we have
$$\|\Delta\tilde\varphi\|_{L^{1+\e_\ast}(\tilde S)}\leq C(n,\lambda,\Lambda)$$
for some $\e_\ast=\e_{\ast}(n,\lambda,\Lambda)>0$ depending only on $n,\lambda$ and $\Lambda$.
Using (\ref{Anorm}) and (\ref{detAh}), we find that
\begin{eqnarray*}
 \|\Delta\varphi\|_{L^{1+\e_\ast}(S_{\varphi}(x_0, h))} &\leq&  \|A_h^{-1}\|^2  (\det A_h)^{\frac{2}{n}} \|\Delta\tilde\varphi\circ T^{-1}\|_{L^{1+\e_\ast}(S_{\varphi}(x_0, h))} 
 \\&=&  \|A_h^{-1}\|^2  (\det A_h)^{\frac{2}{n}+\frac{1}{1+\e_\ast} }\|\Delta\tilde\varphi\|_{L^{1+\e_\ast}(\tilde S)}
 \\&\leq&
  C(\lambda,\Lambda, n,\diam(\Omega))h^{-n} h^{1+ \frac{n}{2} \frac{1}{1+\e_\ast}} \leq   C(\lambda,\Lambda, n,\diam(\Omega)) h^{-\alpha_2}
\end{eqnarray*}
where
$$\alpha_2=\alpha_2(\lambda,\Lambda, n)= n-1- \frac{n}{2} \frac{1}{1+\e_\ast}> 0.$$
(ii) Rescaling as in \eqref{def:tildes}, we have for $x\in S_{\varphi}(x_0, h)$
\begin{equation}D\varphi(x) = (\det A_h)^{\frac{2}{n}} (A_h^{-1})^t D \tilde\varphi (T^{-1}x).
 \label{DD1}
\end{equation}
Suppose that $x, y\in S_{\varphi}(x_0, h/2)$. Then $T^{-1}x, T^{-1}y\in S_{\tilde \varphi}(\tilde x_0, (\det A_h)^{-2/n} h)$. Applying 
the $C^{1,\alpha}$ estimate for the 
Monge-Amp\`ere equation, due to Caffarelli \cite{C1}, to $\tilde\varphi$, we have
\begin{equation}|D \tilde \varphi (T^{-1}x)-D \tilde \varphi (T^{-1}y)|\leq C(\lambda,\Lambda, n)|T^{-1}x-T^{-1}y|^{\alpha}.
 \label{DD2}
\end{equation}
where $\alpha=\alpha (n,\lambda,\Lambda) \in (0,1)$.
In terms of the function $\varphi$, we infer from (\ref{DD1}) and (\ref{DD2}) that 
\begin{equation}|D\varphi(x)-D\varphi(y)|\leq C(\lambda,\Lambda, n) (\det A_h)^{2/n} \|A_h^{-1}\|^{1+\alpha}|x-y|^{\alpha}.
 \label{DD3}
\end{equation}
Using the volume estimates (\ref{vol-sec1}), we obtain from (\ref{DD3}) and (\ref{Anorm})
the following estimate: 
$$|D\varphi(x)-D\varphi(y)|\leq C(\lambda,\Lambda, n, \diam(\Omega)) h^{-\alpha_1}\||x-y|^{\alpha}~\text{for all }x, y\in  S_{\varphi}(x_0, h/2) $$
where
$\alpha_1= -1+\frac{n}{2}(\alpha+1)>0.$
\end{proof}

{\bf Acknowledgments.} The author would like to thank the anonymous referee for the pertinent comments and the careful reading of the paper 
together with his/her constructive criticisms.

\bibliographystyle{plain}

\end{document}